\documentclass{amsart}
\usepackage{graphicx} 
\usepackage{todonotes} 

\usepackage{graphicx} 
\usepackage{amsmath}
\usepackage{amstext}
\usepackage{mathtools}
\usepackage{amssymb}
\usepackage{textcomp}
\usepackage{listings}
\usepackage{graphicx}
\usepackage{physics}
\usepackage{relsize}
\usepackage{exscale}
\usepackage{amsthm}
\usepackage{amsthm,amsmath}
\usepackage{amsmath,mathdots}
\usepackage{bbm}
\usepackage{mathrsfs}
\usepackage{tikz-cd}
\usepackage{colonequals}
\usepackage{centernot}
\usepackage{thmtools, thm-restate}
\usepackage{refcount}
\newcommand{\s}[1]{\textup{\textsf{#1}}}

\renewcommand{\L}{\langle L,<\rangle}
\newcommand{\Rr}{\langle \R,<\rangle}

\newcommand{\Lzero}{\langle L_0,<_0\rangle}
\newcommand{\Lone}{\langle L_1,<_1\rangle}
\newcommand{\Ln}{\langle L_n, <_n\rangle}

\newcommand{\R}{\mathbb{R}}

\newcommand{\Z}{\mathbb{Z}}

\newcommand{\im}{\mathbin{\hbox{\tt\char'42}}}
\newcommand{\cf}{\text{cf}}

\newcommand{\dom}{\text{dom}}

\newcommand{\twoalphalex}{\langle {}^\alpha 2,<_{\operatorname{lex}}\rangle}
\newcommand{\twoalpha}{{}^\alpha 2}
\newcommand{\lessalpha}{{}^{<\alpha}2}
\newcommand{\twoomegalex}{\langle {}^\omega 2,<_{\operatorname{lex}}\rangle}

\newcommand{\twokappalex}{\langle {}^\kappa 2, <_{\operatorname{lex}}\rangle}
\newcommand{\twokappa}{{}^\kappa 2}
\newcommand{\lex}{<_\text{lex}}
\newcommand{\ccint}{\mathbf{cc}^\zeta}
\newcommand{\ccomega}{\mathbf{cc}^\omega}
\newcommand{\ccomegastar}{\mathbf{cc}^{\omega^*}}
\newcommand{\twoomegaonelex}{\langle {}^{\omega_1} 2,<_{\operatorname{lex}}\rangle}
\newcommand{\otp}{\textnormal{otp}}

\renewcommand{\dom}{\textnormal{dom\,}}
\newcommand{\D}{\mathscr D}

\newcounter{thmcount}

\newtheorem{prop}[thmcount]{Proposition}
\newtheorem{defn}[thmcount]{Definition}
\newtheorem{lemma}[thmcount]{Lemma}

\newtheorem{obs}[thmcount]{Observation}

\theoremstyle{definition}

\newcounter{claimcount}
\numberwithin{claimcount}{thmcount}
\newtheorem{claim}[claimcount]{Claim}
\newcommand*{\claimproofname}{Proof of claim}
\newenvironment{claimproof}[1][\claimproofname]{\begin{proof}[#1]}{\end{proof}}
\newtheorem{subfact}[claimcount]{Fact}
\newcounter{questioncount}
\newtheorem{q}[questioncount]{Question}
\newtheorem{problem}[questioncount]{Problem}

\counterwithin{thmcount2}{section}

\newtheorem{thm*}{Theorem}

\usepackage{cite}
\usepackage[shortlabels]{enumitem}
\usepackage[new]{old-arrows}
\usepackage{slashed}

\newcounter{scratch}

\definecolor{thilo}{cmyk}{0.1,0,0.8,0}

\makeatletter
\@namedef{subjclassname@2020}{
\textup{2020} Mathematics Subject Classification}
\makeatother

\title[IEPRs on Higher Analogues of the Real Line]{Infinite-Exponent Partition Relations \\on Higher Analogues of the Real Line}
\author{Lyra A.\ Gardiner}
\address{\textnormal{Lyra A.\ Gardiner}\\Department of Pure Mathematics and Mathematical Statistics \& Trinity College, University of Cambridge}
\email{lag44@cam.ac.uk}

\author{Jonathan Schilhan}
\address{\textnormal{Jonathan Schilhan}, University of Vienna,
Institute of Mathematics,
Kurt Gödel Research Center,
Kolingasse 14-16,
1090 Vienna,
Austria}
\email{jonathan.schilhan@univie.ac.at}

\author{Thilo Weinert}
\address{\textnormal{Thilo Weinert}}
\email{thilo.weinert@univie.ac.at}

\date{}

\subjclass[2020]{Primary 03E02; Secondary 06A05, 03E25, 05D10}
\keywords{Linear orders, partition relations, Ramsey theory, Axiom of
Choice}

\begin{document}
\maketitle


We present a number of results concerning infinite-exponent partition relations (IEPRs) on linear orders of the form $\twoalphalex$ for $\alpha$ an ordinal, generalising the setting of the real line (i.e.\ the case $\alpha = \omega$), working throughout in \s{ZF}. As a particular consequence of our results, we obtain a full classification of the relation $\twoalphalex \rightarrow (\tau)^\tau$ for $\tau$ countable.
\section{Introduction}\label{introduction:section}
\subsection{Background and main results}\label{background:subsection}
This paper is a sequel to the first author's \cite{ieprsonr}, in which infinite-exponent partition relations on arbitrary linear orders were introduced and a number of results were proved for the setting of the real number line $\Rr$. In this paper we expand our scope to ``higher analogues of the reals", by which we mean linear orders of the form $\twoalphalex$ for $\alpha$ an arbitrary infinite ordinal, proving a number of results in this setting by combining ideas from that paper with ideas from the third author's \cite{crtolo}. The majority of the results in this paper were obtained during a visit by the first and second authors to the third author in Udine in March 2024, and a subsequent visit by the first author to the second author in Vienna in September 2025.

Our base theory throughout this paper is \s{ZF} without the Axiom of Choice, as IEPRs on linear orders are either false or trivial in a certain sense under \s{AC}; see \cite[\S 1.3]{ieprsonr} for a discussion of this. Our main results are the following; they are all stated in terms of the \emph{arrow notation} for partition relations, defined in \S\ref{terminologynotation:subsection}.
\begin{restatable}{thm}{tauplustau}\label{tauplustau:thm}
    Let $\tau \ne 0$ be an order type with $\tau + \tau \le \tau$. Then for any ordinal $\alpha$, \[\twoalphalex\centernot \rightarrow (\tau)^\tau.\]
\end{restatable}

\begin{restatable}{thm}{omegaomegastar}\label{omegaomegastar:thm}
    Let $\alpha$ be an ordinal and $\tau$ a scattered, well-orderable order type. If $\omega\omega^* \le \tau$ or $\omega^*\omega \le \tau$ then
    \[\twoalphalex \centernot \rightarrow (\tau)^\tau.\]
\end{restatable}
\begin{restatable}{thm}{similartor}\label{similartor:thm}
    For $\alpha$ a countable ordinal, $\tau$ a countably infinite order type, if $\tau$ is not of the form $\omega + k$ or $k + \omega^*$ for $k$ a natural number then
    \[\twoalphalex \centernot\rightarrow (\tau)^\tau,\]
    and moreover, if $\tau$ is of this form, then
    \[\twoalphalex \rightarrow (\tau)^\tau \iff \omega \rightarrow (\omega)^\omega.\]
\end{restatable}

\begin{restatable}{thm}{affordable}\label{affordable:thm}
    Let $\kappa$ be an initial ordinal and let $\tau$ be a finite sum of ordinals and reverse ordinals $< \kappa$. Then
    \[\twokappalex \rightarrow (\tau)^\tau \iff \kappa \rightarrow (\beta)^\beta,\]
    where $\beta = \beta(\tau)$ is a certain ordinal defined in terms of $\tau$.
\end{restatable}

From the above we deduce the following corollary:

\begin{restatable}{cor}{trichotomy}\label{trichotomy:cor} \textnormal{\textbf{(Trichotomy)}}
Let $\tau$ be a countably infinite order type. Then one of the following holds:
\begin{enumerate}
    \item $\tau$ is of the form $\omega + k$ or $k + \omega^*$ for some $k \in \omega$, and
    \[\twoomegalex \rightarrow (\tau)^\tau \iff \omega \rightarrow (\omega)^\omega;\]
    \item $\tau$ is not of the above form but it is a finite sum of ordinals and reverse ordinals, $\twoalphalex\centernot \rightarrow (\tau)^\tau$ for $\alpha$ countable, but
    \[\twoomegaonelex \rightarrow (\tau)^\tau \iff \omega_1 \rightarrow (\beta)^\beta\]
    for $\beta = \beta(\tau)$ an ordinal depending on $\tau$;
    \item $\tau$ is not a finite sum of ordinals and reverse ordinals, and for all $\alpha$
    \[\twoalphalex \centernot \rightarrow (\tau)^\tau.\]
\end{enumerate}
\end{restatable}
\subsection{Terminology and notation}\label{terminologynotation:subsection}
Our notational conventions largely follow \cite{ieprsonr} and \cite{crtolo}. We reserve the letters $\sigma, \tau$, and $\varphi$ for order types, which are isomorphism classes of linear orders; for $\sigma$ an order type, $\sigma^*$ denotes its reverse. For $\L$ a linear order, $\otp\L$ denotes its order type. For $\sigma,\tau$ order types, $\sigma \le \tau$ denotes the statement that any order of type $\tau$ has a suborder of type $\sigma$. We say that an order type $\tau$ is countable or well-orderable to mean that if $\L$ is an order of type $\tau$, then $L$ is countable or well-orderable, respectively. The operation $+$ will denote concatenation of orders, so $\sigma + \tau$ is the order type of a linear order consisting of a copy of $\sigma$ followed by a copy of $\tau$; multiplication is colexicographic, so $\sigma\tau$ is the order type of $\tau$-many copies of $\sigma$.

$\zeta = \omega^* + \omega$ denotes the order type of the integers, $\eta$ the order type of the rationals, and $\lambda$ the order type of the real line, all under their usual orderings. Parentheses $()$ and brackets $[]$ will be used to denote open and closed intervals in the usual way; for $\L$ a linear order and $x \in L$,
\[(x,\rightarrow) \coloneqq \{y \in L : x < y\},\]
and analogously for $[x,\rightarrow)$, $(\leftarrow,x)$, and $(\leftarrow,x]$.

Throughout the paper we will be considering orders of the form $\twoalphalex$ for $\alpha$ an infinite ordinal; $\lex$ is the \emph{lexicographic order}, defined by
\[x \lex y \iff \{x(\delta) < y(\delta)\},\]
where $\delta = \delta(x,y)$ is the minimal ordinal $< \alpha$ for which $x(\delta) \neq y(\delta)$. In the special case $\alpha = \omega$, $\twoomegalex$ is the Cantor space under its usual ordering. It is in this sense that we consider the $\twoalphalex$ with $\alpha > \omega$ to be ``higher analogues of the real line".\footnote{The usual real line $\Rr$ is not in fact isomorphic to $\twoomegalex$, but the two orders are embeddable in each other, and as such satisfy the same partition relations; see the discussion of monotonicity after the definition of the partition relation.}

When considering $\twoalphalex$ for some infinite ordinal $\alpha$, we write $\mathscr B : \alpha \to \left|\alpha\right|$ to stand for some fixed bijection from $\alpha$ to its cardinality. Elements of ${}^{\le\alpha}2$ will be thought of as sequences; we write $s \sqsubseteq t$ to mean that $s$ is an initial segment of $t$, and denote the concatenation ``$s$ followed by $t$" by $s ^\frown t$. For $x, y \in {}^{\le\alpha}2$, $x \neq y$, we write $\Delta(x,y)$ for the \textit{maximal common initial segment} of $x$ and $y$, i.e.\ that element of ${}^{<\alpha}2$ of maximal length s.t.\ $\Delta(x,y) \sqsubseteq x,y$; by extension, for $A \subseteq \twoalpha$, write $\Delta(A)$ for the maximal element of ${}^{<\alpha}2$ such that $\Delta(A) \sqsubseteq x$ for all $x \in A$. For $s \in {}^{<\alpha}2$, write $h(s)$ for the height/length of $s$, an ordinal $< \alpha$, so for $x \neq y$ in $\twoalpha$ we have that $\delta(x,y) = h\circ \Delta (x,y)$ is the height of the split between $x$ and $y$.  We use the set-theoretic notation for images of functions: for $f$ a function, $X \subseteq \dom f$, 
    \[f \im X \coloneqq \{f(x) : x \in X\}.\]

For $\L$ a linear order, $\sigma$ an order type, write $[\L]^\sigma$ to mean the set of all subsets of $L$ which are ordered as $\sigma$ in the induced suborder. We will frequently abbreviate this to $[L]^\sigma$ when the order $<$ is understood. For $\L$ a linear order, $\sigma$, $\tau$ order types with $\tau \le \sigma \le \otp\L$, and $\chi$ a set, the \textit{partition relation}
\[\L \rightarrow (\sigma)^\tau_\chi\]
is the statement that for any $F : [L]^\tau \to \chi$ (thought of as a ``colouring of the copies of $\tau$ in $\L$") there exists $H \in [L]^\sigma$ which is \textit{homogeneous} (or \textit{monochromatic}) for $F$, in the sense that
\[\left|F \im [H]^\tau\right| = 1.\]
The negation of such a relation is written with $\centernot \rightarrow$ instead of $\rightarrow$, and when the colour set $\chi = 2 = \{0,1\}$ it is usually omitted from the notation. Partition relations satisfy monotonicity in the following sense: if $\L \rightarrow (\sigma)^\tau_\chi$ for some $\L$, $\sigma$, $\tau$, $\chi$, and $\langle L',<'\rangle$, $\sigma'$, $\chi'$ are such that $\otp\L \le \otp \langle L',<'\rangle$, $\sigma' \le \sigma$, and $\left|\chi'\right| \le \left|\chi\right|$, then it follows that $\langle L',<'\rangle \rightarrow (\sigma')^\tau_{\chi'}$. We do not in general have monotonicity in the \textit{exponent} $\tau$.\footnote{In general, monotonicity in the exponent fails, e.g.\ provably $\Rr \centernot \rightarrow (\omega + \omega)^\omega$, but $\Rr \rightarrow (\omega + \omega)^{\omega + 1}$ in Solovay's model (see \cite[\S4]{ieprsonr}); in \S\ref{uncountable:section} we will see another failure with exponents $3$ and $\zeta$. We can, however, recover a certain approximation to monotonicity in the exponent; see Observation \ref{coherentexponentlift:obs} in \S\ref{countable:section}.} Our main object of study in this paper is partition relations of the form
\[\twoalphalex \rightarrow (\tau)^\tau\]
and their negations, for $\tau$ an infinite order type. Note by the remark above about monotonicity that if this ``minimal" relation fails, i.e.\ $\twoalphalex \centernot \rightarrow (\tau)^\tau$, then every other relation with exponent $\tau$ fails, i.e.\ $\twoalphalex \centernot \rightarrow (\sigma)^\tau_\chi$ for all $\sigma$ and $\chi$.

For $\L$ a linear order, $A \subseteq L$, a \emph{condensation class} in $A$ is an equivalence class under the equivalence relation $x \sim y \ratio\Longleftrightarrow$ ``there are finitely many elements of $A$ between $x$ and $y$ (in the ordering $<$)". These are the fibres of the finite condensation $\mathbf{c}_F$; see \cite[p.\ 79]{rosenstein} for details. A condensation class is either finite or it is ordered as one of $\omega$, $\omega^*$, or $\zeta$. For $\sigma \in \{\omega,\omega^*,\zeta\}\cup \omega$, write $\mathbf{cc}^\sigma(A)$ for the set of condensation classes of $A$ of order type $\sigma$.

A linear order $\L$ is \emph{scattered} if no subset of $L$ is densely ordered in the induced suborder; equivalently, if $L$ is well-orderable, $\L$ is scattered if $[L]^\eta = \emptyset$. The scattered orders with which we concern ourselves in this paper will all be in bijection with an ordinal, allowing us to appeal to Hausdorff's analysis of them from \cite{hausdorfforders}; cf. e.g.\ \cite[\S5.3]{rosenstein}.

A linear order $\L$ is \emph{(additively) indecomposable} if whenever it is decomposed into an initial segment $\langle A,<\rangle$ and a final segment $\langle B,<\rangle$, $\L$ embeds in one of $\langle A,<\rangle$ or $\langle B,<\rangle$; $\L$ is \emph{strictly indecomposable to the left} (resp., \emph{right}) if it is always the case that $\L$ embeds in $\langle A,<\rangle$ and not in $\langle B,<\rangle$ (resp.\ that $\L$ embeds in $\langle B,<\rangle$ and not in $\langle A,<\rangle$).
\subsection{Useful definitions and observations}\label{definitions:subsection}
Our first definition fleshes out an idea which was used implicitly in the proof of \cite[Lemma 14]{ieprsonr}.
\begin{defn}\label{dense:defn}
    For $\D \subseteq [\L]^\tau$, say $\D$ is \emph{dense} in $[\L]^\tau$ (by analogy with forcing) if for any $A \in [\L]^\tau$, some subcopy of $\tau$ in $A$ is in $\D$, i.e.\ $[A]^\tau \cap \D \neq \emptyset$. Continuing the analogy, say that $\D$ is \emph{open dense} if it is dense and downwards closed, i.e.\ has the property that if $A \in \D$ then $[A]^\tau \subseteq \D$.
\end{defn}
The usefulness of this notion comes from the following observation:
\begin{obs}\label{denseisenough:obs}
    Let $\L$ be a linear order, $\tau$ an order type, and $\chi$ a set. If $\D$ is dense in $[\L]^\tau$ and $F : \D \to \chi$ has no homogeneous set, in the sense that for any $A \in \D$ there is some $A' \in [A]^\tau \cap \D$ with $F(A') \neq F(A)$, then $F$ can be extended to a colouring defined on the whole of $[\L]^\tau$ with no homogeneous set. In particular, to show $\L \centernot \rightarrow (\tau)^\tau_\chi$ it is sufficient to find a dense set $\D \subseteq [\L]^\tau$ and a colouring $F : \D \to \chi$ with no homogeneous set.
\end{obs}
\begin{proof}
    Let $\D$, $F$ be as given, and extend $F$ to a colouring $F' : [\L]^\tau \to \chi$ like so; fix some $c \in \chi$, and for $A \in [\L]^\tau$, set
    \[F'(A) = \begin{cases*}
        F(A) & if $A \in \D$;\\
        c & otherwise.
    \end{cases*}\]
    We claim no $A \in [\L]^\tau$ can be homogeneous for $F'$: given any $A$, we can reduce to some $A' \in [A]^\tau \cap \D$; by our assumption on $F$, there is some $A'' \in [A']^\tau \cap \D$ with $F(A') \neq F(A'')$, i.e.\ $F'(A') \neq F'(A'')$, so $F'$ takes multiple values on $[A]^\tau$.
\end{proof}
\begin{defn}\label{canonised:defn}
    For $A \in [\twoalphalex]^\omega$, say that $A$ is \emph{canonised} if whenever $x,y,z \in A$ are such that $x < y$ and $x < z$,
    \[\delta(x,y) = \delta(x,z);\]
    equivalently if for $x < y$ elements of $A$, the quantity $\delta(x,y)$ is uniquely determined by $x$. Symmetrically, for $A \in [\twoalphalex]^{\omega^*}$, say $A$ is canonised if for $x < y$ in $A$, $\delta(x,y)$ is uniquely determined by $y$. Finally, for $A \in [\twoalphalex]^\zeta$, say $A$ is canonised if it consists of two pieces $A_0 \in [\twoalphalex]^{\omega^*}$ and $A_1 \in [\twoalphalex]^\omega$ such that $A_0 < A_1$, each $A_i$ is canonised, and moreover whenever $x \in A_0$ and $y \in A_1$,
    \[\delta(x,y) = \min \{\delta(x',y') : x',y' \in A\}.\]
\end{defn}
The above definition is perhaps better understood visually: for $\sigma \in \{\omega,\omega^*,\zeta\}$, $A \in [\twoalphalex]^\sigma$ is canonised if its splitting type is the following:
\\

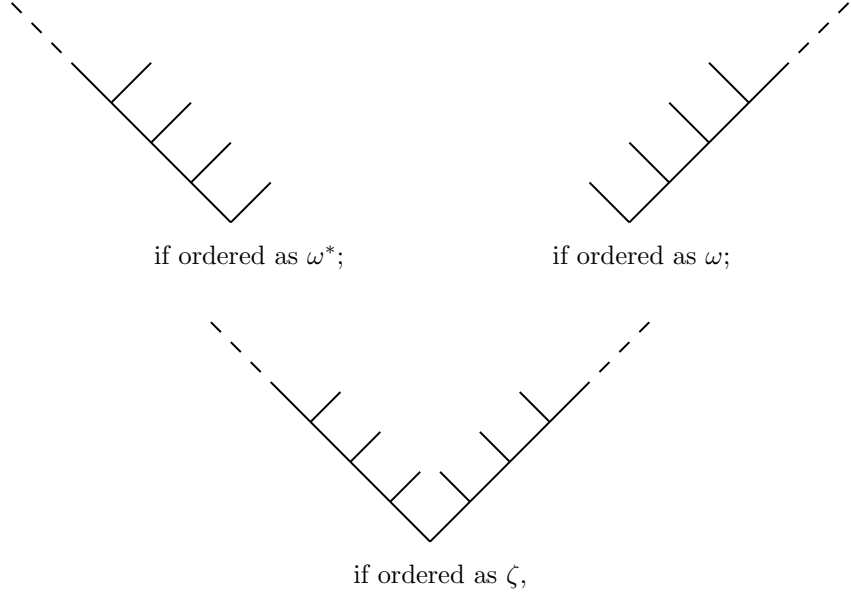
\tikzset{every picture/.style={line width=0.75pt}} 

\begin{tikzpicture}[x=0.75pt,y=0.75pt,yscale=-1,xscale=1]

\draw    (120,120) -- (200,200) ;
\draw    (220,180) -- (200,200) ;
\draw    (200,160) -- (180,180) ;
\draw    (180,140) -- (160,160) ;
\draw    (160,120) -- (140,140) ;
\draw    (110,110) -- (115,115) ;
\draw    (100,100) -- (105,105) ;
\draw    (90,90) -- (95,95) ;

\draw (480,120) -- (400,200) ;
\draw (400,200) -- (380,180) ;
\draw (420,180) -- (400,160) ;
\draw (440,160) -- (420,140) ;
\draw (460,140) -- (440,120) ;
\draw (485,115) -- (490,110) ;
\draw (495,105) -- (500,100) ;
\draw (505,95) -- (510,90) ;

\draw (160,210) node [anchor=north west][inner sep=0.75pt]   [align=left] {if ordered as $\omega ^*$;};
\draw (360,210)node [anchor=north west][inner sep=0.75pt]   [align=left] {if ordered as $\omega$;};

\draw (220,280) -- (300,360) ;
\draw (380,280) -- (300,360) ;
\draw (280,340) -- (295,325) ;
\draw (260,320) -- (275, 305) ;
\draw (240,300) -- (255, 285) ;
\draw (210, 270) -- (215, 275) ;
\draw (200, 260) -- (205, 265) ;
\draw (190, 250) -- (195, 255) ;
\draw (385,275) -- (390,270) ;
\draw (395, 265) -- (400,260) ;
\draw (405, 255) -- (410, 250) ;
\draw (320,340) -- (305,325) ;
\draw (340,320) -- (325,305) ;
\draw (360,300) -- (345,285) ;
\draw (260,370)node [anchor=north west][inner sep=0.75pt]   [align=left] {if ordered as $\zeta$,};
\end{tikzpicture}\\
using visualisations of splitting types as in \cite{crtolo}; cf. Definition 1.3.2(5) and Figure 1 in that paper.
\begin{lemma}\label{canonisation:lemma}
    Let $\alpha$ be an ordinal. There is a map $\mathscr F : \mathcal P(\twoalpha) \to \mathcal P(\twoalpha)$ with the property that for any order type $\tau$ and any $A \in [\twoalphalex]^\tau$, $\mathscr F(A) \in [A]^\tau$ and every infinite condensation class in $\mathscr F(A)$ is canonised.
\end{lemma}
\begin{proof}
    Observe first that for since for any $x < y < z$, \[\delta(x,z) = \min\{\delta(x,y),\delta(y,z)\},\] we have that given any $\sigma \in \{\omega,\omega^*,\zeta\}$ and $Z \in [\twoalphalex]^\sigma$, the minimal value $\min \{\delta(a,b) : a, b \in Z\}$ must in fact be attained by two consecutive $a < b$ in $Z$. Now let $A \in [\twoalphalex]^\tau$ for some $\tau$. We find $\mathscr F(A) \in [A]^\tau$ by reducing the infinite condensation classes of $A$ to subcopies of the same order type which are canonised. For $X = \{x_n : n \in \omega\} \in \ccomega(A)$, ordered ascendingly, we define a subset $X' \coloneqq \{x_{n_k} : k \in \omega\}$ as follows: for each $k \in \omega$, let $n_k$ be that value of $n \in \omega$ which minimises $\delta(x_{n_k}, x_{n_k + 1})$ subject to $n_k > n_l$ for all $l < k$. Then for any $k < m$, $\delta(x_{n_k}, x_{n_m}) = \delta(x_{n_k},x_{n_k + 1})$, so in particular $X' \in [X]^\omega$ is canonised. For $X \in \ccomegastar(A)$, we apply the symmetric procedure to obtain $X' \in [X]^{\omega^*}$ which is canonised. For $X \in \ccint(A)$, let $x^{(0)} < x^{(1)}$ be the consecutive elements of $X$ such that $\delta(x^{(0)}, x^{(1)})$ is minimal; then write
    \begin{align*}
        X_0 &\coloneqq \{x \in X : x \le x^{(0)}\};\\
        X_1 & \coloneqq \{x \in X : x \ge x^{(1)}\}.
    \end{align*}
    We now apply the procedures described above to obtain $X_0' \in [X_0]^{\omega^*}$, $X_1' \in [X_1]^\omega$, and set $X' = X_0' \cup X_1'$. Then $X' \in [X]^\zeta$ and $X'$ is canonised. We let $\mathscr F(A)$ be the subset of $A$ obtained by replacing each infinite condensation class $X$ of $A$ with its corresponding $X'$. Clearly $\mathscr F(A) \subseteq A$ and every infinite condensation class of $\mathscr F(A)$ is canonised. Since we have obtained $\mathscr F(A)$ simply by replacing intervals of $A$ by subsets of those intervals of the same order type, it follows that $\mathscr F(A)$ is also ordered as $\tau$, as required.
\end{proof}
The significance of the existence of this function is that it gives us a canonical way to shrink any $A \subseteq \twoalpha$ to a subcopy of the same order type all of whose infinite condensation classes are canonised, which will be very useful in some of our proofs.
\section{Relations which are negative for all $\alpha$}\label{arbitrary:section}
We begin with two results which show that some order types $\tau$ are such that provably $\twoalphalex \centernot \rightarrow (\tau)^\tau$ for all ordinals $\alpha$.

\tauplustau*
We note here that any such $\tau$ is necessarily non-scattered.\footnote{One can e.g.\ fix an order-embedding witnessing that $\tau + \tau + \tau \le \tau$ and iterate it on a point to find a  copy of $\eta$ in $\tau$. This fact was proved by Ginsburg and Morel, independently of each other, in \cite{ginsburgremarks} and \cite{morelarithmetic}, respectively.} For $\tau$ countable, this is in fact an equivalence; if $\tau$ is non-scattered, then $\eta \le \tau$, so any countable order embeds in $\tau$; in particular, $\tau + \tau \le \tau$.
\begin{proof}\setcounter{scratch}{\value{thmcount}}
\setcounter{thmcount}{\getrefnumber{tauplustau:thm}}
    For $A \subseteq \twoalpha$, $s \in {}^{<\alpha}2$, write 
    \begin{align*}
        A^0_s &\coloneqq \{x \in A: s^\frown\langle 0 \rangle \sqsubseteq x\}\\
        A^1_s & \coloneqq \{x \in A: s^\frown\langle 1 \rangle \sqsubseteq x\}.
    \end{align*}
    For $A \in [\twoalphalex]^\tau$, $s \in \lessalpha$, say $s$ is $\tau$-\textit{splitting for} $A$ if $A^0_s$ and $A^1_s$ both embed $\tau$.

    \begin{claim}\label{uniquemintausplitting:claim}
        For any $A \in [\twoalphalex]^\tau$, there is a unique $s \in \lessalpha$ of minimal length which is $\tau$-splitting for $A$.
    \end{claim}
    
    \begin{claimproof}
        First note that given any $s <_\text{lex} t$ both $\tau$-splitting for $A$, their meet $\Delta (s,t)$ is also $\tau$-splitting for $A$. This is because $A^0_{\Delta (s,t)} \supseteq A^0_s$, and $A^1_{\Delta (s,t)} \supseteq A^1_t$, and $A^0_s, A^1_t$ each embed $\tau$ by assumption. So given that there are any $\tau$-splitting nodes for $A$, there is a unique one of minimal length; it remains to show that there are any such nodes at all.

    Consider the relation $\sim_\tau$ defined on $A$ by $x \sim_\tau y$ iff $A \cap [x,y]$ does not embed $\tau$. Note that this is an equivalence relation, as if $x < y < z$ and $A \cap [x,z]$ does embed $\tau$, then it also embeds e.g.\ $\tau + \tau + \tau$ and so at least one or the other of $A\cap [x,y]$ or $A\cap [y,z]$ must embed $\tau$; for $x \in A$, write $[x]$ for the equivalence class of $x$ under $\sim_\tau$. We have moreover that $\sim_\tau$ is a condensation, i.e.\ the equivalence classes are convex subsets of $A$, and as such the ordering on $A$ induces an ordering on these equivalence classes. We claim that this is a dense order: given $[x] < [y]$, by definition the interval between $x$ and $y$ in $A$ embeds $\tau$; but since $\tau + \tau \le \tau$, we have that $\tau + 1 + \tau \le \tau$, so we may fix a copy of $\tau + 1 + \tau$ between $x$ and $y$ in $A$, and let $z$ be the element of it corresponding to the $1$ (we remark that this $z$ is not necessarily unique, but this is not a problem). Then $x \not \sim_\tau z$ and $z \not \sim_\tau y$, so $[x] < [z] <[y]$.

    Now consider the set \[\begin{split}
        S \coloneqq\{s \in \lessalpha\,:& \,\exists \,x, y \in A \text{ with } [x] < [y],\\ & \,s = \Delta (x, y),\text{ and } [x], [y]\text{ not extremal in }A/\sim_\tau\},
    \end{split}\]
    where here by \textit{extremal} we mean maximal or minimal. Note that if $s, t$ are both in $S$ then so is $\Delta (s,t)$. It follows that the element of $S$ of minimal length is unique; call this $s_A$. We claim that $s_A$ is $\tau$-splitting for $A$. Let $x < y$ witness that $s_A \in S$, so $s_A = \Delta (x, y)$ and $[x] < [y]$. Then since $[x]$ and $[y]$ are not extremal, there exist $x', y' \in A$ with $[x']< [x]$ and $[y]<[y']$, and $[x'],[y']$ also not extremal. Then by minimality $s_A = \Delta(x', y')$ also. But now, both $x$ and $x'$ extend $s_A^\frown \langle 0 \rangle$ and both $y$ and $y'$ extend $s_A ^\frown \langle 1 \rangle$; in particular, $A^0_{s_A} \supseteq A \cap [x',x]$ and $A^1_{s_A} \supseteq A \cap [y,y']$; but since $[x'] < [x] < [y] < [y']$, we have in particular that $x' \not \sim_\tau x$ and $y \not \sim_\tau y'$, so $A \cap [x',x]$ and $A \cap [y,y']$ both embed $\tau$.
    \end{claimproof}

    We now build an injection $f_A: {}^{<\omega}2 \to \lessalpha$ which preserves both the tree structure and the lexicographic ordering of ${}^{<\omega}2$ by means of the following recursion:

    \[\begin{split}
			f_A(\emptyset) &= s_A,\text{ and for }i \in \{0,1\},\text{ given }f_A(t)\text{ for some } t \in {}^{<\omega}2,\\
			f_A(t^\frown \langle i \rangle) &\text{ is the minimal-height }\tau\text{-splitting node for }A^i_{f_A(t)}.
		\end{split}
		\]
    Then in particular $f_A(t ^\frown \langle i \rangle)$ extends $f_A(t)^\frown \langle i \rangle$, so $f_A$ preserves both the tree structure and the lexicographic ordering of ${}^{<\omega}2$, as claimed.
    
    Now we define a colouring $F : [\twoalphalex]^\tau \to 2$ by, for $A \in [\twoalphalex]^\tau$,
    \[F(A) = \begin{cases*}
        0 & if $h(f_A(\langle 0 \rangle)) \ge h(f_A(\langle 1 \rangle))$;\\
        1 & if $h(f_A(\langle 0 \rangle)) < h(f_A(\langle 1 \rangle))$.
    \end{cases*}\]
    \begin{claim}\label{foranytwotausplittings:claim}
        For any $s, t \in \lessalpha$ which are both $\tau$-splitting for $A$ and have $s \lex t$, there is $B \in [A]^\tau$ with $f_{B}(\langle 0 \rangle) = s$, $f_{B}(\langle 1 \rangle) = t$, and $f_{B}(\emptyset) = \Delta (s,t)$.
    \end{claim}
    \begin{claimproof} By definition each of $A^0_s$, $A^1_s$, $A^0_t$, and $A^1_t$ embed $\tau$, and since $s$ and $t$ do not extend each other, all four of these sets are disjoint. Since $\tau + \tau + \tau + \tau \le \tau$, we can find some $\tau_0, \tau_1, \tau_2, \tau_3$, all bi-embeddable with $\tau$, such that $\tau = \tau_0 + \tau_1 + \tau_2 + \tau_3$. Then let $B$ be formed of the disjoint union of a copy of $\tau_0$ in $A^0_s$, a copy of $\tau_1$ in $A^1_s$, a copy of $\tau_2$ in $A^0_t$, and a copy of $\tau_3$ in $A^1_t$.

    Since all of $B$ extends $\Delta (s,t)$ and $B^0_{\Delta (s,t)}$, $B^1_{\Delta (s,t)}$ both embed $\tau$ (because $\Delta (s,t)^\frown \langle 0 \rangle \sqsubseteq s$ and $\Delta (s,t) ^\frown \langle 1 \rangle \sqsubseteq t$), it follows that $\Delta (s,t)$ is the minimal $\tau$-splitting node for $B$. Then since every element of $B^0_{\Delta (s,t)}$ extends $s$ and $s$ is $\tau$-splitting for $B$ by construction, $f_B(\langle 0 \rangle) = s$, and similarly $f_B(\langle 1 \rangle) = t$.
    \end{claimproof}

    Now, using Claim \ref{foranytwotausplittings:claim}, we will show that no $A \in [\twoalphalex]^\tau$ can be homogeneous for the colouring $F$ defined above. First observe that if $A$ is such that $h(f_A(\langle 0 \rangle)) = h(f_A(\langle 1 \rangle))$, then, applying Claim \ref{foranytwotausplittings:claim} with e.g.\ $s = f_A(\langle 0 \rangle)$, $t = f_A(\langle 11 \rangle)$, we obtain some $B \in [A]^\tau$ with $h(f_B(\langle 0 \rangle)) < h(f_B(\langle 1 \rangle))$, and so $F(A) \neq F(B)$.
    
    It follows that for some $A \in [\twoalphalex]^\tau$ to be homogeneous for $F$, it must be the case either that $h(f_B(\langle 0 \rangle)) > h(f_B(\langle 1 \rangle))$ for all $B \in [A]^\tau$, or that $h(f_B(\langle 0 \rangle)) < h(f_B(\langle 1 \rangle))$ for all $B \in [A]^\tau$. But now, for any $s \lex t$ in ${}^{<\omega}2$, we can apply Claim \ref{foranytwotausplittings:claim} to $f_A(s)$ and $f_A(t)$, and obtain either that for every $s \lex t \in {}^{<\omega}2$, $h(f_A(s)) > h(f_A(t))$, or that for every $s \lex t \in {}^{<\omega}2$, $h(f_A(s)) < h(f_A(t))$. Both situations are impossible, as ${}^{<\omega}2$ contains both $\omega$-sequences and $\omega^*$-sequences in $\lex$, so in either case we would get an infinite descending sequence of ordinals.
\end{proof}\setcounter{thmcount}{\value{scratch}}


\omegaomegastar*
We first prove a useful characterisation, which will be of importance again in Section \ref{uncountable:section}:
\begin{lemma}\label{finitesumchar:lemma}
    Let $\tau$ be a well-orderable order type with $\omega\omega^* \not \le \tau$ and $\omega^* \omega \not \le \tau$. Then $\tau$ is a finite sum of ordinals and reverse ordinals.
\end{lemma}
\begin{proof}
    Such a $\tau$ is necessarily scattered. We will make use of iterations of the finite condensation $\mathbf{c}_F$; see \cite[pp.\ 79--80]{rosenstein} for details. The $\beta$-condensation class of $x$ in $\L$ will refer to the set of all points mapped to the same point as $x$ in the $\beta^\text{th}$ iteration of the finite condensation $\mathbf{c}_F$, i.e.\ the set denoted $\mathbf{c}^\beta(x)$ in \cite{rosenstein}.
    
    We proceed by induction on Hausdorff rank: suppose the statement is true for all $\sigma$ of lower Hausdorff rank than $\tau$. Let $\L$ be ordered as $\tau$, and write $\alpha$ for the Hausdorff rank of $\L$ (i.e.\ of $\tau$).
    If $\alpha = \alpha' + 1$ is a successor, then $\L$ can be written as a sum
    \[\L = \sum_{n \in \Z} \langle L_n,<\rangle,\]
    where each $\langle L_n,<\rangle$ has Hausdorff rank at most $\alpha'$ and some of the terms in this sum may be empty. If cofinally many of the $\langle L_n,<\rangle$ embed $\omega^*$, then $\L$ embeds $\omega^*\omega$, a contradiction by our assumption on $\tau$, so there is some $n^+ \in \Z$ such that $\langle L_n,<\rangle$ is well-ordered for every $n \ge n^+$, and so $\sum_{n \ge n^+} \langle L_n,<\rangle$ is isomorphic to an ordinal; similarly, there is $n^- \in \Z$ such that $\sum_{n \le n^-} \langle L_n,<\rangle$ is isomorphic to the reverse of an ordinal. Wlog $n^- \le n^+$. Then
    \[\L = \sum_{n \le n^-}\langle L_n,<\rangle + \sum_{n^- < n < n^+} \langle L_n,<\rangle + \sum_{n \ge n^+} \langle L_n,<\rangle,\]
    where the first and last term in the sum are a reverse well-order and a well-order, respectively, and the term in the middle is a finite sum of terms which by the inductive hypothesis can each be written as finite sums of well-orders and reverse well-orders. It follows that this is also true of $\L$.
    
    Otherwise, $\alpha$ is a limit. Fix some arbitrary $x \in L$, and write \[L = L^- \cup \{x\} \cup L^+,\] where $L^-$ is the part below $x$ and $L^+$ is the part above $x$. Now further split $L^-$ and $L^+$ into the following pieces:
    \[L^{\pm} = \bigcup_{\beta < \alpha} L^{\pm}_\beta,\]
    where for each $\beta < \alpha$, $L^{\pm}_\beta$ is the part of $L^\pm$ which is contained in the $\beta$-condensation class of $x$, but not the $\gamma$-condensation class of $x$ for any $\gamma < \beta$. We remark that for a given $\beta < \alpha$ one or the other of $L^-_\beta$, $L^+_\beta$ may be empty, but not both.
    
    We now decompose $\L$ as a sum
    \[\L = {\sum_{\beta < \alpha}}^* \langle L_\beta^-,<\rangle + \langle\{x\},<\rangle + \sum_{\beta < \alpha} \langle L_\beta^+,< \rangle,\]
    where here the $\Sigma^*$ indicates that the first sum is anti-well-ordered, i.e.\ it is a sum indexed by $\alpha^*$. By the same rationale as above, there is some $\beta^+ < \alpha$ for which $\langle L^+_\beta,<\rangle$ does not embed $\omega^*$ for all $\beta \ge \beta^+$, and some $\beta^- < \alpha$ for which $\langle L^-_\beta,<\rangle$ does not embed $\omega$ for all $\beta \ge \beta^-$. Then there are ordinals $\gamma$, $\delta$ with
    \[\L \cong \delta^* + \left({\sum_{\beta < \beta^-}}^* \langle L_\beta^-,<\rangle + \langle\{x\},<\rangle + \sum_{\beta < \beta^+} \langle L_\beta^+,< \rangle\right) + \gamma,\]
    and since the middle term ${\sum_{\beta < \beta^-}}^* \langle L_\beta^-,<\rangle + \langle\{x\},<\rangle + \sum_{\beta < \beta^+} \langle L_\beta^+,< \rangle$ is contained in a $\max(\beta^-,\beta^+)$-condensation class of $\L$, it has Hausdorff rank at most $\max(\beta^-,\beta^+) < \alpha$. It follows by the inductive hypothesis that $\L$ can be written as a finite sum of well-orders and reverse well-orders.
\end{proof}
\begin{proof}[Proof of Theorem \ref{omegaomegastar:thm}]
    \setcounter{scratch}{\value{thmcount}}
    \setcounter{thmcount}{\getrefnumber{omegaomegastar:thm}}
    The colouring we use to witness this is in fact the same for every such $\tau$; let $C: \mathcal P(\twoalpha) \to 2$ be given by, for $A \subseteq \twoalpha$,
    \[C(A) = \begin{cases*}
        0 & if whenever $x < y$ are in an infinite condensation class of $A$,\\& there is some $z \in A \setminus \{x,y\}$ with $\Delta(x,y) \sqsubseteq z$;\\
        1 & otherwise.
    \end{cases*}\]
    Equivalently, $C(A) = 1$ iff there is some $s \in \lessalpha$ such that $A \cap [s]$ is a set of size $2$ whose elements live in an infinite condensation class of $A$. We first observe the following two properties of $C$:
    \begin{claim}\label{colour0:claim}
        If $A = \mathscr F(A')$ for some $A'$, then $C(A) = 0$, where $\mathscr F$ is the canonisation map defined in Lemma \ref{canonisation:lemma}.\footnote{In fact, it can be seen that this is an equivalence, but the given definition of $C$ simplifies the rest of the argument.}
    \end{claim}
    \begin{claimproof}
        If $A = \mathscr F(A')$ for some $A'$ then for any $x < y$ in an infinite condensation class $K$, either there are infinitely many $z \in K$ below $x$, and $\Delta(x,y) \sqsubseteq z$ for all such $z$, or there are infinitely many $z \in K$ above $y$, and again $\Delta(x,y) \sqsubseteq z$ for all such $z$.
    \end{claimproof}
    \begin{claim}\label{inhomogoninterval:claim}
        Let $\varphi,\tau$ be order types with $\varphi,\tau \le\twoalphalex$ such that $\varphi$ embeds convexly in $\tau$, i.e.\ in any copy of $\tau$ there is an interval ordered as $\varphi$. If $C$ witnesses $\twoalphalex \centernot \rightarrow (\varphi)^\varphi$ then it also witnesses $\twoalphalex \centernot \rightarrow (\tau)^\tau$.
    \end{claim}
    \begin{claimproof}
        Write $\tau = \sigma_0 + \varphi + \sigma_1$ for some $\sigma_0$, $\sigma_1$, and assume without loss of generality that $\sigma_0$ does not end with a finite condensation class which is part of an infinite condensation class of $\tau$, and symmetrically for $\sigma_1$.\footnote{Formally, if $\varphi$ is of the form $\omega + \varphi'$ and $\sigma_0$ is of the form $\sigma_0'+ n$ for some finite $n$ and $\sigma_0'$ with no maximal element, replace $\sigma_0$ by $\sigma_0'$, and similarly if $\varphi$ is of the form $\varphi' + \omega^*$ and $\sigma_1$ is of the form $n + \sigma_1'$ for some $\sigma_1'$ with no minimal element, replace $\sigma_1$ by $\sigma_1'$.} Given $A \in [\twoalphalex]^\tau$, fix a decomposition of $A$ as $A_0 \cup A^* \cup A_1$, where the $A_i$ are ordered as the $\sigma_i$ and $A^*$ is ordered as $\varphi$; reduce to $A' \in [A]^\tau$ given by
        \[A' \coloneqq \mathscr F(A_0) \cup A^* \cup \mathscr F(A_1).\]
        Then $C(A') = C(A^*)$. It follows that if $A^*$ is not homogeneous for $C$, there is some $A^\dagger \in [A^*]^\varphi$ with $C(A^\dagger) \neq C(A^*)$; then, defining $A''\in [A']^\tau$ by \[A'' \coloneqq \mathscr F(A_0) \cup A^\dagger \cup \mathscr F(A_1),\] we have $C(A'') \neq C(A')$.
    \end{claimproof}
    We will prove the following statement by induction on Hausdorff rank:
    \[
        \omega\omega^* \le \tau \text{ or }\omega^*\omega \le \tau \implies C \text{ witnesses } \twoalphalex \centernot \rightarrow (\tau)^\tau. \tag{$\ast$}
    \]
    Suppose there is some well-orderable scattered order for which $(*)$ does not hold; let $\tau$ be such an order type of minimal Hausdorff rank. Then in particular $\omega\omega^* \le \tau$ or $\omega^*\omega \le \tau$. We will appeal to a result of Laver and a result of Jullien:\footnote{We note that both results were proved in \s{ZFC}, but even without examining the proofs to see where Choice was used, we see that both still hold in \s{ZF}; given some $\tau$ scattered and well-orderable, we may reduce to some inner model $W$ of \s{ZFC} containing a well-ordering of an order of type $\tau$, and apply Fact \ref{sumofindecs:fact} in $W$ to find a decomposition of $\tau$ into indecomposable$^W$ pieces $\tau = \tau_0 + \tau_1 + \dots + \tau_n$; each $\tau_k$ is strictly indecomposable to either the right or the left in $W$, by Fact \ref{indecisstrict:fact}, and so must remain so in $V$, as it is still scattered there.}
    \begin{subfact}\label{sumofindecs:fact} \textnormal{(\!\!\cite{laver})}
        Any scattered linear order can be written as a finite sum of indecomposable scattered linear orders.
    \end{subfact}
    \begin{subfact}\label{indecisstrict:fact}
    \textnormal{(\hspace{-0.001em}\cite{jullienthesis})}
        A scattered order is indecomposable iff it is either strictly indecomposable to the right or strictly indecomposable to the left.
    \end{subfact}
    By Claim \ref{inhomogoninterval:claim} and Fact \ref{sumofindecs:fact}, we may assume that $\tau$ is indecomposable. It then follows by Fact \ref{indecisstrict:fact} that $\tau$ 
    is either strictly indecomposable to the right, i.e.\ embeds in all of its final segments, or strictly indecomposable to the left. Wlog it is strictly indecomposable to the right. Let $\kappa = \cf(\tau)$, which necessarily exists as $\tau$ is well-orderable; then fix a decomposition of $\tau$ as
    \[\tau = \sum_{\vartheta < \kappa} \varphi_\vartheta.\]
    
    These $\varphi_\vartheta$ are then well-orderable scattered orders of lower Hausdorff rank than $\tau$, so by our minimality assumption, they satisfy $(*)$. In particular, if $\omega\omega^* \le \varphi_\vartheta$ or $\omega^*\omega \le \varphi_\vartheta$ for some $\vartheta$, $C$ witnesses $\twoalphalex \centernot \rightarrow (\varphi_\vartheta)^{\varphi_\vartheta}$, and so by Claim \ref{inhomogoninterval:claim} it witnesses $\twoalphalex \centernot \rightarrow (\tau)^\tau$, so $\tau$ satisfies $(*)$, a contradiction. It therefore must be the case that $\omega\omega^* \not\le \varphi_\vartheta$ and $\omega^*\omega \not\le \varphi_\vartheta$ for each $\vartheta$, so the $\varphi_\vartheta$ are finite sums of ordinals and reverse ordinals by Lemma \ref{finitesumchar:lemma}. It is therefore impossible for $\omega\omega^* \le \tau$, so it must be the case that $\omega^*\omega \le \tau$, and in particular this requires that cofinally many of the $\varphi_\vartheta$ contain infinite reverse ordinals, and so contain intervals ordered as $\omega^*$ (which are either condensation classes or are included in condensation classes ordered as $\zeta$).\footnote{If instead $\tau$ was strictly indecomposable to the left, we would have $\omega^*\omega \not\le \tau$ and $\omega\omega^* \le \tau$.}
    
    Now let $A \in [\twoalphalex]^\tau$, and wlog $A = \mathscr F(A)$, so $C(A) = 0$; we will show that we can find $A' \in [A]^\tau$ with $C(A') = 1$. Let $x_0 < x_1$ be consecutive elements of some interval of $A$ ordered as $\omega^*$ with the property that $\delta(x_0,x_1)$ is minimal amongst \[\{\delta(a,b): a < b\text{ consecutive members of an interval of }A\text{ ordered as }\omega^*\}.\]
    Now, since there are cofinally many intervals in $\tau$ ordered as $\omega^*$, we may find some later such interval, $O$, say; then we may find $y_{-1} < y_0 < y_1$ consecutive elements of $O$ such that $\delta(y_0,y_1)$ is minimal in
    \[\{\delta(a,b) : a, b \in O\}.\]
    Then $\delta(y_{-1},y_0) > \delta(y_0,y_1) \ge \delta(x_0,x_1)$ by minimality. We are now ready to define our $A' \in [A]^\tau$. Write $\tau = \tau_0 + \tau_1$, where $\tau_0$ is the order type of $A \cap (\leftarrow,x_1)$ and $\tau_1$ is the order type of $A \cap [x_1,\rightarrow)$. Then since $\tau$ is strictly indecomposable on the right, we may find a copy of $\tau_1$ in the final segment $A \cap (y_0,\rightarrow)$ of $A$, $A_1'$, say; then set
    \[A' \coloneqq (A\cap (\leftarrow,x_1)) \cup \{y_{-1},y_0\} \cup A_1'.\]

    We claim that $C(A') = 1$, and that this is witnessed by $y_{-1},y_0$. Since $\Delta(y_{-1},y_0) \sqsupseteq \Delta(y_0,y_1)^\frown\langle0\rangle $ and for any $z \in A'$ with $z > y_0$, \[z \ge_\text{lex} y_1 \sqsupseteq \Delta(y_0,y_1)^\frown \langle 1\rangle,\] no $z > y_0$ extends $\Delta(y_{-1},y_0)$. By minimality of $\delta(x_0,x_1)$, we have that $x_0 \not \sqsupseteq \Delta(y_0,y_1)$, so since any $z \in A'$ with $z < y_{-1}$ has $z \le x_0$, it follows that any such $z$ has $z \not \sqsupseteq \Delta(y_0,y_1)$, and in particular $z \not \sqsupseteq \Delta(y_{-1},y_0)$. Thus, the only elements of $A'$ extending $\Delta(y_{-1},y_0)$ are $y_{-1}$ and $y_0$, and so $C(A') = 1$.
\end{proof}\setcounter{thmcount}{\value{scratch}}

\section{The behaviour of $\twoalphalex$ for $\alpha$ countable}\label{countable:section}

The results in this section can be summarised as follows:
\similartor*
We therefore have that the behaviour of the relation $\L \rightarrow (\tau)^\tau$ with $\tau$ countable is identical in the setting $\L = \Rr$ and the setting $\L = \twoalphalex$ for $\alpha$ any other countable ordinal; cf.\ \cite[Theorem 1]{ieprsonr}. We remark here that the ordinal IEPR $\omega \rightarrow (\omega)^\omega$ enjoys a special significance in the study of IEPRs on ordinals; it is known to be consistent with \s{ZF} relative to an inaccessible cardinal, by work of Mathias in \cite{mathias}, building on work of Solovay in \cite{solovay}, and it is also a consequence of the axiom \s{AD}$_\R$, the \emph{Axiom of Determinacy for games on} $\R$, by a result of Příkrý in \cite{prikryadr}. It is open whether this inaccessible is necessary, i.e.\ whether $\omega \rightarrow (\omega)^\omega$ has consistency strength greater than that of \s{ZF}; it is also open whether it follows from \s{AD}, which is weaker than \s{AD}$_\R$.

We prove Theorem \ref{similartor:thm} by means of a number of lemmas. 
Recall that we write $\mathscr B : \alpha \to \omega$ to refer to some fixed bijection, and let us write $\mathscr N(x,y) = \mathscr B(\delta(x,y))$ for the natural number corresponding to the level $< \alpha$ of the split between the nodes $x,y \in \twoalpha$.
We will show the negative direction of Theorem \ref{similartor:thm} (i.e.\ the statement that if $\tau$ is a countable type not of the form $\omega + k$ or $k + \omega^*$, then $\twoalphalex \centernot \rightarrow (\tau)^\tau$) by means of a series of lemmas, together with Theorem \ref{tauplustau:thm}. 
We first introduce two key notions which will help us in building colourings with no homogeneous sets.
\begin{defn}\label{coherentselector:defn}
    For $\L$ a linear order and $\sigma \le \tau$ order types, a partial function
    \[f: [\L]^\tau \rightharpoonup [\L]^\sigma\]
    is called a \emph{coherent selector} if it satisfies the following conditions for all $A \in \dom f$:
    \begin{enumerate}[(a)]
        \item $f(A) \subseteq A$;
        \item For all $B \in [f(A)]^\sigma$, $A' \coloneqq (A\setminus f(A)) \cup B \in \dom f$, and moreover $f(A') = B$.
    \end{enumerate}
\end{defn}
The existence of coherent selectors allows us to get a sort of approximation to monotonicity in the exponent; coherent selectors pick out subcopies of a ``smaller" order type in a sufficiently uniform way that we can lift negative relations:
\begin{obs}\label{coherentexponentlift:obs}
    Let $\L$ be a linear order, $\sigma \le \tau$ order types and $\chi$ a set. If \[\L \centernot \rightarrow (\sigma)^\sigma_\chi\] and there is a coherent selector
    \[f : [\L]^\tau \rightharpoonup [\L]^\sigma\]
    such that $\dom f$ is dense in $[\L]^\tau$, then $\L \centernot \rightarrow (\tau)^\tau_\chi$.
\end{obs}
\begin{proof}
    Let $F: [\L]^\sigma \to \chi$ be a colouring with no homogeneous set; then $F \circ f$ is a colouring defined on a dense subset of $[\L]^\tau$ (namely $\dom f$) with no homogeneous set.
\end{proof}
A further strengthening of this notion of coherent selectors will allow for the direct construction of colourings with no homogeneous set.
\begin{defn}\label{mutuallycoherent:defn}
    Let $\L$ be a linear order and $\sigma_0, \sigma_1, \tau$ order types with $\sigma_0,\sigma_1 \le \tau$. Coherent selectors $f_0 : [\L]^\tau \rightharpoonup [\L]^{\sigma_0}$ and $f_1 : [\L]^\tau \rightharpoonup [\L]^{\sigma_1}$ are said to be \emph{mutually coherent} if the following conditions hold for each $i$:
    \begin{enumerate}[(a)]
        \item $\dom f_0 = \dom f_1$;
        \item For $A \in \dom f_i$, $f_0(A) \cap f_1(A) = \emptyset$,
    \end{enumerate} and in addition, given $A \in \dom f_i$ and writing $A'$ for the result of replacing each of $f_0(A)$, $f_1(A)$ by some subcopies $B_0 \in [f_0(A)]^{\sigma_0}$, $B_1 \in [f_1(A)]^{\sigma_1}$, i.e.\ \[A' \coloneqq (A \setminus (f_0(A) \cup f_1(A))) \cup (B_0 \cup B_1),\] we have that
    \begin{enumerate}[(a)]
        \setcounter{enumi}{2}
        \item $A' \in \dom{f_i}$;
        \item $f_i(A') = B_i$.
    \end{enumerate}
\end{defn}
The significance of this definition comes from the following result:
\begin{lemma}\label{twomutuallycoherent:lemma}
    Let $\alpha$ be a countable ordinal and let $\tau$ be an order type such that, for some $\sigma_0, \sigma_1 \in \{\omega,\omega^*\}$ there exists a pair of mutually coherent selectors $f_0: [\twoalphalex]^\tau \rightharpoonup [\twoalphalex]^{\sigma_0}$ and $f_1: [\twoalphalex]^\tau \rightharpoonup [\twoalphalex]^{\sigma_1}$ whose (common) domain is dense in $[\twoalphalex]^\tau$. Then \[\twoalphalex \not \rightarrow (\tau)^\tau.\]
\end{lemma}
\begin{proof}
    Let $\alpha$, $\tau$ be as given and fix some such $f_0,f_1$. Write $\D = \dom f_0 = \dom f_1$, a dense subset of $[\twoalphalex]^\tau$ by assumption. For $A \in \D$, say $f_0(A) = \{x^A_n : n \in \omega\}$, $f_1(A) = \{y^A_n : n \in \omega\}$, where in each case these enumerations are increasing if $\sigma_i = \omega$ and decreasing if $\sigma_i = \omega^*$. Define a colouring $F: \D \to 2$ by, for $A \in \D$,
    \[F(A) = \begin{cases*}
        0 & if $\mathscr N(x_0^A,x_1^A) \ge \mathscr N(y_0^A,y_1^A)$\\
        1 & if $\mathscr N(x_0^A,x_1^A) < \mathscr N(y_0^A,y_1^A)$.
    \end{cases*}\]
    If we set $A' = A\setminus \{x_0, \dots, x_{k-1}\}$ for some $k$, by the mutual coherence of $f_0$ and $f_1$ we have $x_n^{A'} = x_{n + k}^A$ and $y_n^{A'} = y_n^A$ for all $n \in \omega$; it follows that $A' \in [A]^\tau$ has $\mathscr N(y_0^{A'},y_1^{A'}) = \mathscr N(y_0^A,y_1^A)$ and $\mathscr N(x_0^{A'},x_1^{A'}) = \mathscr N(x_k^A,x_{k+1}^A)$, which takes arbitrarily large natural number values, so with an appropriate choice of $k$ we can find such an $A'$ with $F(A') = 0$. Similarly, by instead setting $A' = A \setminus \{y_0,\dots,y_{k-1}\}$ for some $k$ we can find $A' \in [A]^\tau$ with $F(A') = 1$. It follows that no $A \in \D$ can be homogeneous for $F$, so $\twoalphalex \centernot \rightarrow (\tau)^\tau$ by Observation \ref{denseisenough:obs}.
\end{proof}
\begin{lemma}\label{twoomegageneral:lemma}
    Let $\tau$ be an order type with at least two condensation classes ordered as $\omega$ or $\omega^*$, and let $\alpha$ be countable. Then $\twoalphalex \not \rightarrow (\tau)^\tau$.
\end{lemma}
\begin{proof}
    Let $\tau$ be as stated. For $C \in [\twoalphalex]^\omega$ or $C \in [\twoalphalex]^{\omega^*}$, write $C = \{c_n : n \in \omega\}$, where this enumeration is $\lex$-increasing if $C$ is ordered as $\omega$ and $\lex$-decreasing if $C$ is ordered as $\omega^*$. Now set $\delta(C) \coloneqq \delta(c_0,c_1)$.

    \begin{claim} The set of $A \in [\twoalphalex]^\tau$ satisfying the following conditions is dense in $[\twoalphalex]^\tau$:
    \begin{enumerate}[(a)]
        \item For every $C \in \ccomega(A) \cup \ccomegastar(A)$, $\delta(C) = \min\{\delta(x,y) : x,y \in C\}$;
        \item There exists some $\delta' < \alpha$ such that there are precisely two $C \in \ccomega(A) \cup \ccomegastar(A)$ with $\delta(C) \le \delta'$.
    \end{enumerate}
    \end{claim}
    \begin{claimproof} Let $A \in [\twoalphalex]^\tau$ be arbitrary; we show that we can reduce it to an $A' \in [A]^\tau$ satisfying both (a) and (b). To obtain (a), replace $A$ by $\mathscr F(A)$, where $\mathscr F$ is the canonisation map defined in Lemma \ref{canonisation:lemma}. Now, to obtain (b), given any $C_0,C_1 \in \ccomega(A) \cup \ccomegastar(A)$, set $\delta' = \max\{\delta(C_0),\delta(C_1)\}$, and for each other $C \in \ccomega(A) \cup \ccomegastar(A)$, remove $\{c_0,c_1,\dots,c_{m-1}\}$ for $m = m_C\in \omega$ minimal with the property that $\delta(c_N,c_{N+1}) > \delta'$ for all $N' \ge m$. Note that the resulting set still has all infinite condensation classes canonised, so (a) still holds.
    \end{claimproof}

    Write $\mathscr D$ for the set of $A \in [\twoalphalex]^\tau$ satisfying (a) and (b). Now, for $A \in \mathscr D$, let $C_0, C_1 \in \ccomega(A) \cup \ccomegastar(A)$ be the two condensation classes guaranteed by (b), named such that $C_0 < C_1$. For $i \in \{0,1\}$, write $C_i = \{c^{(i)}_n : n \in \omega\}$, as above. Then the functions $f_0 : \mathscr D \to [\twoalphalex]^{\omega^{(*)}}$, $f_1: \mathscr D \to [\twoalphalex]^{\omega^{(*)}}$ given by $f_i : A \mapsto C_i \setminus \{c^{(i)}_0\}$ are mutually coherent, and so it follows from Lemma \ref{twomutuallycoherent:lemma} that $\twoalphalex \not \rightarrow (\tau)^\tau$.
\end{proof}
\begin{lemma}\label{zetacc:lemma}
    Let $\tau$ be an order type with at least one condensation class ordered as $\zeta$, and let $\alpha$ be countable. Then $\twoalphalex \not \rightarrow (\tau)^\tau$.
\end{lemma}
\begin{proof}
    First we deal with the simplest case, i.e.\ $\tau = \zeta$; the rest of the proof will be devoted to showing that we can always either reduce to this case or appeal directly to Lemma \ref{twomutuallycoherent:lemma}.
    \begin{claim}\label{zetaalphacountable:claim}
        $\twoalphalex \not \rightarrow (\zeta)^\zeta$ for $\alpha$ countable.
    \end{claim}
    \begin{claimproof}
        This proof essentially goes via Lemma \ref{twomutuallycoherent:lemma}, but the colouring we obtain is sufficiently easy to describe that it seems better to include the full details of it. Let $A \in [\twoalphalex]^\zeta$. Then $A$ has an associated successor function, $S_A: A \to A$. Let $x_A \in A$ be that $x \in A$ that minimises $\delta(x,S_A(x))$. Now write $\mathscr N_0(A) \coloneqq \mathscr N(S_A^{-1}(x_A),x_A)$ and $\mathscr N_1(A) \coloneqq \mathscr N(S_A(x_A),S_A^2(x_A))$, and define a colouring $F: [\twoalphalex]^\zeta \to 2$ by, for $A \in [\twoalphalex]^\zeta$,
		\[F(A) = \begin{cases*}
		0 & if $\mathscr N_0(A) \ge \mathscr N_1(A)$\\
		1 & if $\mathscr N_0(A) < \mathscr N_1(A)$.
		\end{cases*}\]
		Observe that for any $y > x_A$ in $A$, $A' \coloneqq A\setminus (x_A,y)$ has $S_{A'}(x_A) = y$ and $\delta(x_A,y) = \delta(x_A,S_A(x_A))$ so this is still minimal and $x_{A'} = x_A$; in this way we can reduce to some $A' \in [A]^\zeta$ with $\mathscr N_0(A') = \mathscr N_0(A)$ and $\mathscr N_1(A')$ equal to $\mathscr N(y,S_A(y))$ for an arbitrary $y > x_A$ in $A$, and so in particular can be made to be an arbitrarily large element of $\omega$. Similarly, taking $A' = A \setminus (y,S_A(x_A))$ for some $y < x_A$ in $A$, we can make $\mathscr N_0$ arbitrarily large while keeping the value of $\mathscr N_1$ the same. It follows that any $A \in [\twoalphalex]^\zeta$ has some subset $A' \in [A]^\zeta$ with $F(A') \neq F(A)$.\footnote{We remark that the colouring $F$ defined in this proof is exactly the colouring we obtain if we appeal to Lemma \ref{twomutuallycoherent:lemma} with the mutually coherent selectors $f_0 : [\twoalphalex]^\zeta \to [\twoalphalex]^{\omega^*}$, $f_1 : [\twoalphalex]^\zeta \to [\twoalphalex]^\omega$ given by $f_0 : A \mapsto \{x \in A: x \le x_A\}$ and $f_1 : A\mapsto \{x \in A: x \ge S_A(x_A)\}.$}
    \end{claimproof}
    The rough idea behind the rest of the proof is as follows: if we can easily identify a single copy of $\zeta$, we can pick that one out and colour it; if not, we can reduce to a situation where we can pick out two copies of $\zeta$ based on their left halves, and play the right halves against each other, in the sense of Lemma \ref{twomutuallycoherent:lemma}.
    
    If $\tau$ has only one condensation class ordered as $\zeta$, then $f: [\twoalphalex]^\tau \to [\twoalphalex]^\zeta$ which sends $A$ to the unique element of $\ccint(A)$ is a coherent selector, and it follows from Claim \ref{zetaalphacountable:claim} and Observation \ref{coherentexponentlift:obs} that $\twoalphalex\not \rightarrow (\tau)^\tau$. So assume $\tau$ has at least two condensation classes ordered as $\zeta$. For $C \in \ccint(A)$, write $\delta(C) = \min\{\delta(x,y):x,y \in C\}$. This in particular equals $\delta(x_C,S_C(x_C))$. Write $\delta_0(C) \coloneqq \delta(S_C^{-1}(x_C),x_C)$, and $\delta_1(C) \coloneqq \delta(S_C(x_C),S_C^2(x_C))$.
    \begin{claim}
        Every $A \in [\twoalphalex]^\tau$ has a subset $B \in [A]^\tau$ with one of the following properties:
        \begin{enumerate}[(a)]
            \item There is a unique $C \in \ccint(B)$ with $\delta(C)$ minimal, \textit{or}
            \item There is some $\delta' < \alpha$ such that among those $C \in \ccint(B)$ with $\delta(C)$ minimal, precisely two of them have $\delta_0(C) \le \delta'$.
        \end{enumerate}
    \end{claim}
    \begin{claimproof}
        Suppose $A \in [\twoalphalex]^\tau$ does not satisfy (a). We will show that there is $B \in [A]^\tau$ satisfying (b). Consider the values of $\delta_0(C)$ attained by those $C \in \ccint(A)$ with $\delta(C)$ minimal. If there is some minimal value attained by multiple such $C$, set $\delta'$ equal to this value; otherwise, let $\delta'$ be the second-smallest value of $\delta_0(C)$ for $C \in \ccint(A)$ with $\delta(C)$ minimal. Now pick out some $C_0, C_1 \in \ccint(A)$ with $\delta(C_0) = \delta(C_1)$ minimal and $\delta_0(C_i) \le \delta'$ for $i \in \{0,1\}$. We shrink $A$ to some $B \in [A]^\tau$ such that $B$ satisfies condition (b), witnessed by $C_0, C_1 \in \ccint(B)$, in the following way: for every $C \in \ccint(A) \setminus \{C_0,C_1\}$, replace $C$ by $C' \in [C]^\zeta$ given by
        \[C' \coloneqq C \setminus (S^{-n}_C(x_C),S_C(x_C))\]
        where $n \in \omega$ is minimal such that $\delta(S_C^{-n-1}(x_C),S_C^{-n}(x_C)) > \delta'$. Then these $C'$ are condensation classes of $B$, and all have $\delta(C') = \delta(C)$ but $\delta_0(C') > \delta'$; further, $C_0$ and $C_1$ are still condensation classes of $B$, but now are the only two elements of $\ccint(B)$ with $\delta_0(C) \le \delta'$.
    \end{claimproof}
    It follows that the subset of elements of $[\twoalphalex]^\tau$ satisfying either (a) or (b) is dense, so it suffices to define a colouring on these $A \in [\twoalphalex]^\tau$. If $A \in [\twoalphalex]^\tau$ satisfies (a), i.e.\ there is a unique $C \in \ccint(A)$ minimising $\delta(C)$, then we can pick out this $C$ and colour it as in the proof of Claim \ref{zetaalphacountable:claim}. Otherwise, $A$ satisfies (b); let $C_0$, $C_1 \in \ccint(A)$ be the two condensation classes guaranteed by the statement of (b), named so that $C_0 < C_1$; then $A \mapsto \{x \in C_0: x > x_{C_0}\}$, $A \mapsto \{x \in C_1 : x > x_{C_1}\}$ are mutually coherent selectors defined on those $A \in [\twoalphalex]^\tau$ satisfying (b) and we are done by Lemma \ref{twomutuallycoherent:lemma}.\footnote{Formally, we can phrase this proof as a definition of a single pair of coherent selectors whose domain is the set of those $A \in [\twoalphalex]^\tau$ satisfying (a) or (b), with the functions themselves defined by cases according to whether $A$ satisfies (a) or (b).}
\end{proof}
Our final lemma before we prove Theorem \ref{similartor:thm} is more general than the setting of $\alpha$ countable.
\begin{lemma}\label{kappa:lemma}
    Let $\kappa$ be a regular cardinal and $\alpha$ an ordinal with $\left|\alpha\right| = \kappa$. Then
    \[\twoalphalex \rightarrow (\kappa)^\kappa \iff \kappa \rightarrow (\kappa)^\kappa.\]
\end{lemma}
\begin{proof}
    The reverse implication is trivial, noting simply that $\kappa$ embeds in any such $\twoalphalex$ as a linear order, so we focus on the forwards implication.
    %
    We extend our notation from above, so $\mathscr B : \alpha \to \kappa$ is a fixed bijection and $\mathscr N(x,y) = \mathscr B(\delta(x,y))$, and for $A = \langle a_\beta : \beta < \kappa \rangle \in [\twoalpha]^\kappa$, write
    \[\begin{split}
        \delta(A) &\coloneqq \langle \delta(a_\beta,a_{\beta+1}) : \beta < \kappa\rangle,\text{ and}\\
        N(A) &\coloneqq \langle \mathscr N(a_\beta,a_{\beta+1})) : \beta < \kappa\rangle.
    \end{split}\]
    \begin{claim}\label{kappanincreasing:claim} There is an open dense $\D \subseteq [\twoalpha]^\kappa$ with the property that for all $C \in \D$, $\delta(C)$ and $N(C)$ are both increasing.
    \end{claim}
    \begin{claimproof} Let $\D$ be the set of all $C = \langle c_\beta: \beta < \kappa\rangle \in [\twoalpha]^\kappa$ with the following three properties: 
    \begin{enumerate}[(a)]
        \item $\delta(C)$ is increasing; 
        \item $N(C)$ is increasing;
        \item for any $\beta < \gamma < \kappa$, $\delta(c_\beta,c_\gamma) = \delta(c_\beta,c_{\beta+1}).$
    \end{enumerate}
    We will show that this $\D$ is open dense. First observe that the conjunction of these three properties is open, so we need only show that $\D$ is dense.
    
    Given $A \in [\twoalpha]^\kappa$, we first reduce to a subsequence $B = \langle b_\beta: \beta < \kappa\rangle \in [A]^\kappa$ with the property that $\delta(B)$ is increasing and that for all $\beta < \gamma < \kappa$, the value of $\delta(b_\beta, b_\gamma)$ is determined only by $\beta$, i.e.\ properties (a) and (c) above. We do this iteratively by simultaneously building a sequence $\langle b_\beta : \beta < \kappa\rangle$ and a sequence of ``reservoir sets" $\langle A^{(\beta)} : \beta < \kappa\rangle$ such that each $A^{(\beta)}$ is a final segment of $A$, each $b_\beta \in A^{(\beta)}$, and for $\beta_0 < \beta_1$, $A^{(\beta_0)} \supseteq A^{(\beta_1)}$.

    Set $A^{(0)} = A$. Given $A^{(\beta)}$, we find some $x_\beta, y_\beta \in A^{(\beta)}$ with $x_\beta < y_\beta$ and with the property that $\delta(x_\beta,y_\beta)$ is minimal amongst all pairs in $A^{(\beta)}$.\footnote{Since $A^{(\beta)}$ is well-ordered we can do this without Choice.} By minimality and the fact that $\delta(x,z) = \min(\delta(x,y),\delta(x,z))$ whenever $x < y < z$, we have that $\delta(x_\beta,z) = \delta(x_\beta,y_\beta)$ for every $z \in A^{(\beta)}$ with $z \ge y_\beta$. Now set $b_\beta \coloneqq x_\beta$ and $A^{(\beta+1)} \coloneqq \{z \in A^{(\beta)}: y_\beta \le z\}$. At limit stages $\gamma$, set $A^{(\gamma)} \coloneqq \bigcap_{\beta < \gamma} A^{(\beta)}$; since $\kappa$ is regular, this is a non-empty final segment of $A$. By construction, our two requirements on $B$ hold. We now reduce $B$ to a further subsequence $C$ with $N(C)$ increasing.

    For each $\beta < \kappa$, we have an associated quantity $\delta_\beta = \delta(b_\beta,b_\gamma)$ for any $\gamma > \beta$. Then $\beta \mapsto \mathscr B(\delta_\beta)$ is an injection from $\kappa$ to itself, so there is a $\kappa$-sized subset of $\kappa$ on which it is increasing; set $C$ to be the corresponding elements of $B$. Then $C$ has the property that $\delta(C)$ and $N(C)$ are increasing, and for any $\beta < \gamma < \kappa$, $\delta(c_\beta,c_\gamma) = \delta(c_\beta,c_{\beta+1})$; any $C' \in [C]^\kappa$ also has all of these properties. Since $A$ was arbitrary, we are done.
    \end{claimproof}
    \begin{claim}\label{kappaallsubsetsifnincreasing:claim}
        Let $C \in \D$ as above, and let $X \in [N(C)]^\kappa$. Then there is some $C_X \in [C]^\kappa$ with $N(C_X) = X$.
    \end{claim}
    \begin{claimproof}
        Fix $C$, $X$ as described and define $C_X \coloneqq \langle c_\beta : \mathscr B(\delta(c_\beta,c_{\beta+1})) \in X\rangle$. Then $C_X \in [C]^\kappa$, and for any $c_\beta \in C_X$, writing $c_{\gamma}$ for its successor in $C_X$, we have that $\delta(c_\beta,c_\gamma) = \delta(c_\beta,c_{\beta+1})$ by condition (c) on $\D$, and in particular $\mathscr B(\delta(c_\beta,c_{\beta+1})) \in X$. Since for every $\varepsilon \in X$ there was some $\beta < \kappa$ with $\delta(c_\beta,c_{\beta+1}) = \varepsilon$, we have that $N(C_X) = X$, as required.
    \end{claimproof}
    Now suppose $\kappa \not \rightarrow (\kappa)^\kappa$, and let $F: [\kappa]^\kappa \to 2$ witness this, so $F$ has no homogeneous set. We use this $F$ to induce a colouring $G_F:\D \to 2$ with no homogeneous set like so: for $C \in \D$,
    \[G_F(C) \coloneqq F(N(C)).\]
    But since $C \in \D$, given any $X \in [N(C)]^\kappa$, by Claim \ref{kappaallsubsetsifnincreasing:claim} there is some $C_X \in [C]^\kappa$ with $N(C_X) = X$; thus, if $F$ has no homogeneous set, $G_F$ also has no homogeneous set.
\end{proof}
\begin{proof}[Proof of Theorem \ref{similartor:thm}]
    Let $\tau$ be a countably infinite order type, and let $\alpha$ be a countable ordinal. We proceed by a series of cases. If $\tau$ has at least two condensation classes each ordered as $\omega$ or $\omega^*$, then $\twoalphalex \centernot \rightarrow (\tau)^\tau$ by Lemma \ref{twoomegageneral:lemma}; if $\tau$ has a condensation class ordered as $\zeta$, then  by Lemma $\ref{zetacc:lemma}$; otherwise, $\tau$ has at most one infinite condensation class, ordered as one of $\omega$ or $\omega^*$ if it exists. In any such $\tau$, it is the case that if there are two distinct finite condensation classes, there must be another finite condensation class between them; in particular, if $\tau$ has more than one finite condensation class, it is non-scattered, and so $\tau + \tau \le \tau$; it follows from Theorem \ref{tauplustau:thm} that $\twoalphalex \centernot \rightarrow (\tau)^\tau$ in this case.
    
    The remaining case is that $\tau$ has at most one infinite condensation class, ordered as $\omega$ or $\omega^*$, and at most one finite condensation class, i.e.\ $\tau$ is of the form $\omega + k$ or $k + \omega^*$ for some $k \in \omega$. By Lemma \ref{kappa:lemma},
    \[\twoalphalex \rightarrow (\omega)^\omega \iff \omega \rightarrow (\omega)^\omega,\]
    and by symmetry $\twoalphalex \rightarrow (\omega^*)^{\omega^*} \iff \twoalphalex \rightarrow (\omega)^\omega$.
    
    Now suppose $\tau = \omega + k$ for some $k \in \omega$. If $\twoalphalex \rightarrow (\omega)^\omega$, observe that any $F : [\twoalphalex]^\tau \to 2$ induces a colouring $F' : [\twoalphalex]^\omega \to 2$ by, for $\bar{x} \in [\twoalphalex]^\omega$, writing $\bar{x}'$ for the $\omega$-sequence $\langle \langle 0\rangle ^\frown x_n : n \in \omega\rangle$ and fixing any $k$-tuple $\bar{y} \in [\twoalphalex]^k$ all of whose elements begin with a 1, and setting
    \[F'(\bar{x}) \coloneqq F(\bar{x}'^\frown \bar{y}),\]
    and if $\bar{h}$ is homogeneous for $F'$, then $\bar{h}'^\frown \bar{y}$ is homogeneous for $F$.

    Conversely, if $\twoalphalex \rightarrow (\tau)^\tau$, a colouring $F : [\twoalphalex]^\omega \to 2$ induces a colouring $F'' : [\twoalphalex]^\tau$ by, for $\bar{x} \in [\twoalphalex]^\tau$, simply setting
    \[F''(\bar{x}) \coloneqq F(\bar{x}\restriction \omega).\]
    Then if $\bar{h}$ is homogeneous for $F''$, $\bar{h}\restriction \omega$ is homogeneous for $F$. It follows that 
    \[\twoalphalex\rightarrow (\tau)^\tau \iff \twoalphalex \rightarrow (\omega)^\omega \iff \omega \rightarrow (\omega)^\omega.\]
    By symmetry, if $\tau$ is of the form $k + \omega^*$ for some $k \in \omega$,
    \[\twoalphalex \rightarrow (\tau)^\tau \iff \twoalphalex \rightarrow (\omega^*)^{\omega^*} \iff \omega \rightarrow (\omega)^\omega.\]
\end{proof}
\section{Sums of ordinals and reverse ordinals}\label{uncountable:section}
In this section we characterise the consistency of the relations
\[\twoalphalex \rightarrow (\tau)^\tau\]
for $\tau$ a finite sum of ordinals and reverse ordinals (equivalently, $\tau$ well-orderable and $\omega\omega^*,\omega^*\omega \centernot \le \tau$, by Lemma \ref{finitesumchar:lemma}).

We make use in this section of a notion which extends our usual partition relations. For $\bar{L} = \langle\Lzero, \Lone,\dots, \Ln\rangle$ a sequence of linear orders and $\tau_0, \tau_1, \dots, \tau_n$ order types, write $[\bar{L}]^{\tau_0, \tau_1, \dots, \tau_n}$ as a piece of abbreviated notation for the set $[\Lzero]^{\tau_0} \times [\Lone]^{\tau_1} \times \dots \times [\Ln]^{\tau_n}$. The \emph{polarised partition relation}
\[\begin{pmatrix}
    \Lzero \\ \Lone \\ \vdots \\ \Ln
    \end{pmatrix} \rightarrow \begin{pmatrix}
    \sigma_0 \\ \sigma_1 \\ \vdots \\\sigma_n\end{pmatrix}^{\tau_0,\tau_1,\dots,\tau_n}\]
is the statement that for any colouring $F: [\bar{L}]^{\tau_0, \tau_1, \dots, \tau_n} \to 2$, there is $\bar{H} = \langle H_0, H_1, \dots, H_n\rangle \in [\bar{L}]^{\sigma_0,\sigma_1,\dots,\sigma_n}$ which is homogeneous for $F$, in the sense that $\left|F \im [\bar{H}]^{\tau_0, \tau_1, \dots, \tau_n}\right| = 1$.
\begin{prop}\label{finitesumAD:prop}
    Let $\kappa$ be an uncountable initial ordinal and let $\tau$ be a finite sum of ordinals and reverse ordinals, all $< \kappa$. Then there is an ordinal $\xi = \xi(\tau) < \kappa$ such that
    \[\kappa \rightarrow (\xi)^\xi \implies \twokappalex \rightarrow (\tau)^\tau.\]
\end{prop}
It follows in particular that if $\omega_1 \rightarrow (\alpha)^\alpha$ for all $\alpha < \omega_1$, e.g.\ under \s{AD}, then for all countable $\tau$ with $\omega\omega^* \not \le \tau$ and $\omega^*\omega \not \le \tau$, we have $\twoomegaonelex \rightarrow (\tau)^\tau$.
\begin{proof}
    Fix $\kappa$ and $\tau$. For $\varepsilon$ an ordinal, write $\varepsilon^{(*)}$ to mean either $\varepsilon$ or $\varepsilon^*$, and say that the sum comprising $\tau$ is made up of some alternating pattern of $\beta_0^{(*)}, \beta_1^{(*)}, \dots, \beta_n^{(*)}$, with each $\beta_k$ nonzero. We set $\xi(\tau)$ to equal the ordinal sum $\beta_0 + \beta_1 + \dots + \beta_n$. 
    \begin{claim}\label{polarised:claim} \[\kappa \rightarrow (\xi)^\xi \implies\begin{pmatrix}
    \kappa \\ \kappa \\ \vdots \\ \kappa
    \end{pmatrix} \rightarrow \begin{pmatrix}
    \beta_0 \\ \beta_1 \\ \vdots \\\beta_n\end{pmatrix}^{\beta_0,\beta_1,\dots,\beta_n},\]
    where $\xi = \xi(\tau) = \beta_0 + \beta_1 + \dots + \beta_n$, as above.
    \end{claim}
    \begin{claimproof}
    Assume $\kappa \rightarrow (\xi)^\xi$. For $0 \le k \le n$, write $\xi_k = \sum_{i \le k} \beta_i$, so $\xi = \xi_n$. Given $F: [\kappa]^{\beta_0}\times [\kappa]^{\beta_1}\times \dots \times [\kappa]^{\beta_n} \to 2$ a colouring, observe that this induces a colouring $G_F : [\kappa]^{\xi}\to 2$ in the following way:  for $A = \langle \alpha_\gamma: \gamma < \xi \rangle \in [\kappa]^{\xi}$, set
    \[G_F(A) \coloneqq F(A \restriction \xi_0, A \restriction [\xi_0,\xi_1), \dots, A\restriction [\xi_{n-1},\xi_n)).\]
    Then since by assumption $\kappa \rightarrow (\xi)^{\xi}$, there is some $H = \langle h_\gamma : \gamma < \xi\rangle \in [\kappa]^{\xi}$ homogeneous for $G_F$; then
    \[H' \coloneqq [H \restriction \xi_0]^{\beta_0}\times [H\restriction [\xi_0,\xi_1)]^{\beta_1}\times \dots \times [H\restriction [\xi_{n-1},\xi_n)]^{\beta_n}\]
    is homogeneous for $F$.\end{claimproof}
    
    Let $I = I_0 \cup I_1 \cup \dots \cup I_n \subseteq \twokappa$ be a subset of $\twokappa$ such that $I_i < I_j$ whenever $i < j$ and each $I_k \in [\twokappalex]^{\kappa}$ or $I_k \in [\twokappalex]^{\kappa^*}$, depending on whether $\beta_k$ or $\beta_k^*$ is the $(k+1)^\text{th}$ summand in $\tau$, so $I$ is ordered as an alternating sum of $\kappa$ and $\kappa^*$. Write
    \[(I)^{(\tau)} \coloneqq \left\{A \in [\twokappalex]^\tau: \text{ for }k \le n,\, A \cap I_k \in [I_k]^{\beta_k^{(*)}}\right\},\]
    where in each case $\beta_k^{(*)}$ is interpreted as whichever of $\beta_k$, $\beta_k^{(*)}$ is the $(k+1)^\text{th}$ summand of $\tau$. Observe that if $A \in (I)^{(\tau)}$ and $A' \in [A]^\tau$, then $A' \in (I)^{(\tau)}$ also. In particular, there is a correspondence $\iota : (I)^{(\tau)} \to [\kappa]^{\beta_0}\times [\kappa]^{\beta_1}\times \dots \times [\kappa]^{\beta_n}$; writing $\iota_k : I_k \to \kappa$ for the natural bijection between $I_k$ and $\kappa$, this correspondence is given by
    \[\iota (A) \coloneqq \langle \iota_k (A \cap I_k) : k \le n \rangle;\]
    given any $A \in (I)^{(\tau)}$, $\iota$ maps the subcopies of $\tau$ in $A \in (I)^{(\tau)}$ to the subcopies of $\langle \beta_k : k \le n\rangle$ in $\iota(A)$ coherently, i.e.\ $\iota \im [A]^\tau = [\iota(A)]^{\beta_0,\beta_1,\dots,\beta_n}$. Now, given $F : [\twokappalex]^\tau \to 2$, consider the restriction of $F$ to $(I)^{(\tau)}$. This induces a colouring $G_F : [\kappa]^{\beta_0}\times [\kappa]^{\beta_1}\times \dots \times [\kappa]^{\beta_n} \to 2$ by, for $\bar{A} = \langle A_k : k \le n\rangle$,
    \[G_F(\bar{A}) = F(\iota^{-1}(\bar{A})).\]
    By claim \ref{polarised:claim},
    \[\begin{pmatrix}
    \kappa \\ \kappa \\ \vdots \\ \kappa
    \end{pmatrix} \rightarrow \begin{pmatrix}
    \beta_0 \\ \beta_1 \\ \vdots \\\beta_n\end{pmatrix}^{\beta_0,\beta_1,\dots,\beta_n},\]
    and so there is some $\bar{H} \coloneqq \langle H_k : k \le n\rangle$ homogeneous for $G_F$; then $\iota^{-1}(\bar{H})$ is homogeneous for $F$.
\end{proof}
We remark that the consistency of such relations both demonstrates a sharp contrast between the possible relations on $\twoalphalex$ for $\alpha$ countable and for $\alpha$ uncountable and gives another failure of monotonicity in the exponent; we have from the above that e.g.\ $\twoomegaonelex \rightarrow (\zeta)^\zeta$ is consistent, whereas (a) $\twoalphalex \centernot \rightarrow (\zeta)^\zeta$ for $\alpha$ countable, by Claim \ref{zetaalphacountable:claim}; (b) provably $\twoalphalex \centernot  \rightarrow(\zeta)^3$ for all $\alpha$.\footnote{e.g.\ define $F : [\twoalphalex]^3 \to 2$ by $F (\{x_0,x_1,x_2\}) = 0 \iff \delta(x_0,x_1) > \delta(x_1,x_2)$.}

We now prove a converse to the above. Let $\tau$ be an infinite order type which can be expressed as a finite sum of ordinals and reverse ordinals. We associate an ordinal $\beta(\tau)$ to $\tau$ in the following way: fix $\beta_0, \beta_1,\dots,\beta_{n-1}$ infinite ordinals such that $\tau = \beta_0 + \beta_1^* + \dots + \beta_{n-1}^{(*)}$ or $\tau = \beta_0^* + \beta_1 + \dots + \beta_{n-1}^{(*)}$. For each $i < n$, let
\[\beta_i = \omega^{\varepsilon_0^{(i)}}\cdot k_0^{(i)} + \omega^{\varepsilon_1^{(i)}}\cdot k_1^{(i)} + \dots + \omega^{\varepsilon_{m_i}^{(i)}}\cdot k_{m_i}^{(i)}\]
be the Cantor Normal Form of $\beta_i$. Then we define $\beta(\tau)$ to be the sum of the $\omega^{\varepsilon_j^{(i)}}\cdot k_j^{(i)}$, taken over all possible values of $i$ and $j$, with the summands ordered in (weakly) decreasing order; equivalently, $\beta(\tau)$ is the largest ordinal it is possible to write as a sum of the $\omega^{\varepsilon_j^{(i)}}\cdot k_j^{(i)}$.

\begin{prop}\label{affordableconverse:prop}
    Let $\kappa$ be an uncountable initial ordinal and let $\tau$ be a finite sum of ordinals and reverse ordinals all $< \kappa$. Then
    \[\twokappalex \rightarrow (\tau)^\tau \implies \kappa \rightarrow (\beta)^\beta,\]
    where $\beta = \beta(\tau)$ is the ordinal defined above.
\end{prop}
We first introduce the following definition, which will be very useful in the proof:
\begin{defn}
    Let $\gamma = \omega^\varepsilon$ be an additively indecomposable ordinal and let $\alpha$ be any ordinal. A set $X \in [\twoalphalex]^\gamma$ is said to be \emph{totally canonised} if whenever $x < y < z$ are elements of $X$,
    \[\delta(x,y) = \delta(x,z).\]
    Symmetrically, for $X \in [\twoalphalex]^{\gamma^*}$, $X$ is totally canonised if the quantity $\delta(x,y)$ is determined by $y$.
\end{defn}
\begin{obs}\label{correspondingsubset:obs}
    Define $N : [\twoalphalex]^\gamma \rightharpoonup [\alpha]^\gamma$ like so: $\dom N = \{X \in [\twoalphalex]^\gamma : X\textnormal{ is totally canonised}\}$, and for $X = \{x_\upsilon : \upsilon < \gamma\} \in \dom N$,
    \[N(X) \coloneqq \{\delta(x_\upsilon,x_{\upsilon +1}) : \upsilon < \gamma\}.\] Then for all $X \in \dom N$,
    \begin{enumerate}
        \item $N(X) \in [\alpha]^\gamma$;
        \item For any $Y \in [N(X)]^\gamma$, there is an $X'_Y \in [X]^\gamma$ with $N(X'_Y) = Y$.
    \end{enumerate}
\end{obs}
\begin{proof}
    Let $X \in \dom N$. Since for $x < y < z$, $\delta(x,z) = \min\{\delta(x,y),\delta(y,z)\}$, the fact that $X$ is totally canonised gives that for any $\upsilon < \upsilon' < \gamma$, $\delta(x_\upsilon,x_{\upsilon'}) = \delta(x_\upsilon, x_{\upsilon + 1})$, and $\delta(x_\upsilon,x_{\upsilon+1}) < \delta(x_{\upsilon'},x_{\upsilon'+1})$. It follows that $N(X)$ consists of an increasing $\gamma$-sequence of ordinals $< \alpha$. Now, if $Y \in [N(X)]^\gamma$, set \[X_Y \coloneqq \{x_\upsilon : \delta(x_\upsilon,x_{\upsilon+1}) \in Y\}.\]
    Since for $\upsilon < \upsilon'$, $\delta(x_\upsilon,x_{\upsilon'})$ is determined by $\upsilon$, it follows that $N(X_Y) = Y$.
    \end{proof}
When $\gamma$ is a regular cardinal, any $X \in [\twoalphalex]^\gamma$ can be reduced to some $X' \in [X]^\gamma$ which is totally canonised; for $\gamma$ an arbitrary indecomposable ordinal, this is not necessarily the case, but we can still make use of the concept. Intuitively, the idea behind the proof of Proposition \ref{affordableconverse:prop} is the following: given a colouring $F : [\kappa]^\beta \to 2$, we define an induced colouring $G_F : [\twokappalex]^\tau \to 2$ such that the only possible homogeneous sets for $G_F$ can be divided into indecomposable pieces which are all totally canonised
, and then use the function $N$ defined above to treat such a set as a proxy for an element of $[\kappa]^\beta$; we then apply $F$ to this.
\begin{proof}[Proof of Proposition \ref{affordableconverse:prop}]
    For simplicity of notation we assume that we are in the case that $\tau$ begins with an infinite ordinal; the proof in the case that $\tau$ begins with an infinite reverse ordinal is identical. Let us write $\tau = \beta_0 + \beta_1^* + \dots + \beta_n^{(*)}$, wlog with every $\beta_i$ infinite. Let $A \in [\twokappalex]^\tau$ and wlog every condensation class of $A$ is canonised. Observe that it is almost the case that $A$ decomposes uniquely into an initial segment ordered as $\beta_0$, followed by an interval ordered as $\beta_1^*$, then one ordered as $\beta_2$, etc.: uniqueness fails between every consecutive pair of a reverse ordinal followed by an ordinal, $\beta_i^*$ and $\beta_{i+1}$, say; here there will be an interval ordered as $\omega^* + \omega$, and any consecutive $x < y$ in this interval can be chosen to be the last element of the piece of the decomposition ordered as $\beta_i^*$ and the first element of the piece ordered as $\beta_{i+1}$, respectively. We fix this problem and ensure uniqueness of our decomposition of $A$ by declaring that the consecutive pair $x < y$ will be chosen to minimise $\delta(x,y)$. In what follows we restrict our attention to those $A \in [\twokappalex]^\tau$ for which these $\delta(x,y)$ are minimal across $[A]^\tau$, i.e.\ none of them can be decreased by reducing to any $A' \in [A]^\tau$; this guarantees that whenever $A' \in [A]^\tau$, the piece of $A'$ corresponding to some $\beta_i^{(*)}$ is a subset of the piece of $A$ corresponding to $\beta_i^{(*)}$. This property is dense.
    
    Now, given this decomposition of $A \in [\twokappalex]^\tau$ into pieces ordered as ordinals or reverse ordinals, we can decompose $A$ further into (additively) indecomposable pieces by dividing the piece of $A$ corresponding to some $\beta_i^{(*)}$ into pieces corresponding to each of the indecomposable summands of the Cantor Normal Form of $\beta_i$; this is unique. Refer to the components of this further decomposition of $A$ as the \emph{indecomposable pieces} of $A$.
    
    Let $F: [\kappa]^\beta \to 2$ be a colouring. We will define an induced colouring $G_F: [\twokappalex]^\tau \rightharpoonup 2$ on a dense subset of $[\twokappalex]^\tau$ by means of a series of cases.

    First we define $G_F$ on those $A \in [\twokappalex]^\tau$ with the property that whenever $A' \in [A]^\tau$, some indecomposable piece of $A'$ is not totally canonised. For such $A$, let $X_A \coloneqq \{x_\xi : \xi < \gamma\}$ be the leftmost indecomposable piece of $A$ which is not totally canonised, and set
    \[G_F(A) \coloneqq \begin{cases*}
        0 & if $\delta(x_0,x_1) > \delta(x_1,x_2)$;\\
        1 & if $\delta(x_0,x_1) < \delta(x_1,x_2).$
    \end{cases*}\]
    We show that any such $A$ cannot be homogeneous for this. By reducing $A$ to some appropriate member of $[A]^\tau$, assume without loss of generality that for all $A' \in [A]^\tau$, the leftmost indecomposable piece of $A'$ which is not totally canonised corresponds to the same indecomposable piece of $\tau$. We can now easily reduce to an $A' \in [A]^\tau$ taking either colour: since $X$ is not totally canonised, it is possible to find $\xi_0 < \xi_1 < \xi_2$ with $\delta(x_{\xi_0},x_{\xi_1}) \neq \delta(x_{\xi_0},x_{\xi_2})$, i.e.\ $\delta(x_{\xi_0},x_{\xi_1}) > \delta(x_{\xi_1},x_{\xi_2})$; then \[A' \coloneqq A \setminus \{x_\upsilon : \upsilon <\xi_2 \text{ and }\upsilon \not \in \{\xi_0,\xi_1\}\}\]
    has $G_F(A') = 1$, as its indecomposable piece under consideration is $X \setminus \{x_\upsilon : \upsilon <\xi_2 \text{ and }\upsilon \not \in \{\xi_0,\xi_1\}\}$; meanwhile, $A'' \in [A]^\tau$ obtained by canonising the first condensation class in $X$ has $G_F(A'') = 0$.

    It follows that even with just this partial definition of $G_F$, we know that any $A \in [\twokappalex]^\tau$ which is homogeneous for $G_F$ must have some $A' \in [A]^\tau$ all of whose indecomposable pieces are totally canonised. For each $X$ an indecomposable piece of $A$, write $\varsigma(X) \coloneqq \sup \{\delta(x_\upsilon,x_{\upsilon+1}) : \upsilon < \gamma_X\}$.\footnote{We can run a similar argument to the above and define $G_F$ in such a way that for any $A$ which is homogeneous for $G_F$, the $\varsigma(X)$ are distinct, but this is unnecessary.} We will use these suprema of the splitting levels as offsets, so that we can read off an element of $[\kappa]^{\beta(\tau)}$ from $A$: for $X, Y$ indecomposable pieces of $A$, write $X \prec Y$ if either $\gamma_X > \gamma_Y$ \emph{or} if $\gamma_X = \gamma_Y$ and $X$ is to the left of $Y$ in $\twokappalex$. Then $\prec$ is a total ordering on the indecomposable pieces of $A$. Now write
    \[N'(X,A) \coloneqq \left\{\sum_{X' \prec X}\varsigma(X') + \delta(x_\vartheta,x_{\vartheta + 1}): \vartheta < \gamma_X\right\},\]
    so $N'(X,A)$ is a copy of $N(X)$ shifted by the sum of the suprema of the splitting levels of the indecomposable pieces of $A$ which come before $X$ in the order $\prec$. Set
    \[N'(A) \coloneqq \bigcup \left\{N'(X,A) : X\text{ an indecomposable piece of }A\right\}.\]
    
    Then $N'(A) \in [\kappa]^{\beta(\tau)}$. Now, finally, we are ready to define $G_F$ on these $A$; set
    \[G_F(A) \coloneqq F(N'(A)).\]
    
    Note that any $B \in [N'(A)]^{\beta(\tau)}$ has, for each $X$ an indecomposable piece of $A$, that
    \[B \cap N'(X,A) \in [N'(X,A)]^{\gamma_X}.\]
    By Observation \ref{correspondingsubset:obs}, there is then some $Y \in [X]^{\gamma_X}$ with $N'(Y,(A\setminus X)\cup Y) = B \cap N'(X,A)$; taking the union of these $Y$, it follows that for any $B \in [N'(A)]^{\beta(\tau)}$, there is $A_B \in [A]^\tau$ with $N'(A_B) = B$.

    By assumption, $\twokappalex \rightarrow (\tau)^\tau$, so there is some $H \in [\twokappalex]^\tau$ homogeneous for $G_F$. As we have shown, we may assume without loss of generality that the indecomposable pieces of $H$ are all totally canonised. We claim that $N'(H) \in [\kappa]^{\beta(\tau)}$ is homogeneous for $F$. Let $B \in [N'(H)]^{\beta(\tau)}$; then by the above, there is some $H' \in [H]^\tau$ with $N'(H') = B$. But now
    \[F(B) = G_F(H') = G_F(H) = F(N'(H)),\]
    by the homogeneity of $H$ for $G_F$. It follows that $\kappa \rightarrow (\beta(\tau))^{\beta(\tau)}$.
\end{proof}
\begin{proof}[Proof of Theorem \ref{affordable:thm}]
    Let $\tau$ be a finite sum of ordinals and reverse ordinals, $\beta_0^{(*)}, \beta_1^{(*)}, \dots \beta_n^{(0)}$, say, with each $\beta_k < \kappa$, and let $\xi_n = \beta_0 + \beta_1 + \dots + \beta_n$ as in Proposition \ref{finitesumAD:prop} and $\beta = \beta(\tau)$ the maximum possible sum of the indecomposable pieces of the $\beta_k$ as in Proposition \ref{affordableconverse:prop}. Then by those two results,
    \[\kappa \rightarrow (\xi_n)^{\xi_n} \implies \twokappalex \rightarrow (\tau)^\tau \implies \kappa \rightarrow (\beta)^\beta.\]
    Now, since $\beta \ge \xi_n$, it is also the case that
    \[\kappa \rightarrow (\beta)^\beta \implies \kappa \rightarrow (\xi_n)^{\xi_n}.\]
    It follows that
    \[\kappa \rightarrow (\xi_n)^{\xi_n} \iff \twokappalex \rightarrow (\tau)^\tau \iff \kappa \rightarrow (\beta)^\beta.\]
\end{proof}
A curious consequence of the proof of Theorem \ref{affordable:thm} is that for any $\gamma, \gamma' < \kappa$ whose Cantor Normal Forms have the same first component, we have \[\kappa \rightarrow (\gamma)^\gamma \iff \kappa \rightarrow (\gamma')^{\gamma'}\]
(letting $\gamma = \omega^{\varepsilon_0}\cdot k_0 + \vartheta$, with $\vartheta < \omega^{\varepsilon_0}$, we may set $\tau = \vartheta^*+ \omega^{\varepsilon_0}\cdot k_0$; then $\xi(\tau) = \omega^{\varepsilon_0} \cdot k_0$, and $\beta(\tau) = \gamma$, so \[\kappa \rightarrow (\omega^{\varepsilon_0}\cdot k_0)^{\omega^{\varepsilon_0}\cdot k_0} \iff \twokappalex \rightarrow (\tau)^\tau \iff \kappa \rightarrow (\gamma)^\gamma,\] and similarly for $\gamma'$).

\section{Summary and further work}\label{summaryfurtherwork:section}
\subsection{The trichotomy}\label{trichotomy:subsection}
We can now use the results of the previous sections to characterise the relation $\twoalphalex \rightarrow (\tau)^\tau$ for $\tau$ countable. Recall:
\trichotomy*
\begin{proof}
    Let $\tau$ be a countably infinite order type. If $\tau$ is of the form $\omega + k$ or $k + \omega^*$ for some $k \in \omega$, then by Theorem \ref{similartor:thm}, for any countable $\alpha \ge \omega$,
    \[\twoalphalex \rightarrow (\tau)^\tau \iff \omega \rightarrow (\omega)^\omega.\]
    If $\tau$ is not of this form but is still a finite sum of ordinals and reverse ordinals, then by Theorem \ref{affordable:thm},
    \[\twoomegaonelex \rightarrow (\tau)^\tau \iff \omega_1 \rightarrow (\beta)^\beta,\]
    where $\beta = \beta(\tau)$ is the maximal ordinal expressible as a sum of the indecomposable pieces of the Cantor Normal Forms of the ordinals constituting $\tau$ as an alternating sum, as in Proposition \ref{affordableconverse:prop}.

    Otherwise, $\tau$ is not a finite sum of ordinals and reverse ordinals, i.e.\ $\omega\omega^* \le \tau$ or $\omega^*\omega \le \tau$. Here we split into two cases: if $\tau$ is scattered, then it follows from Theorem \ref{omegaomegastar:thm} that for all $\alpha$,
    \[\twoalphalex \centernot \rightarrow (\tau)^\tau.\]
    If $\tau$ is non-scattered, then $\eta \le \tau \le \eta$, and so $\tau + \tau \le \tau$; it then follows from Theorem \ref{tauplustau:thm} that for all $\alpha$,
    \[\twoalphalex \centernot \rightarrow (\tau)^\tau.\]
\end{proof}
\subsection{Open questions etc.}\label{questions:subsection}
Theorem \ref{similartor:thm} demonstrates that the partition relations on $\twoalphalex$ for $\alpha$ countable are very similar to those on $\R$; it is natural to ask how far this analogy extends.
\begin{q}\label{analogytorrealtypes:q}
    Is it the case that for all $\tau \le \lambda$ and $\alpha \ge \omega$ countable,
    \[\Rr \rightarrow (\tau)^\tau \iff \twoalphalex \rightarrow (\tau)^\tau?\]
\end{q}
\begin{q}\label{analogytorexact:q}
    Does an analogue of \cite[Theorem 2]{ieprsonr} hold for arbitrary countable $\alpha$? That is, is it the case that if all infinite sets of reals are inexact (e.g.\ in Solovay's model and models of \s{AD}), then for all countable $\alpha$, $\twoalphalex \centernot \rightarrow (\tau)^\tau$ if $\tau$ is uncountable?
\end{q}

We can generalise Question \ref{analogytorrealtypes:q} to the following:
\begin{q}
    If $\left|\alpha\right| = \kappa$ and $\tau \le \otp\langle{}^\kappa 2,\lex\rangle$, is it necessarily the case that
    \[\twoalphalex \rightarrow (\tau)^\tau \iff\langle{}^\kappa 2,\lex\rangle \rightarrow (\tau)^\tau?\]
\end{q}

This paper, much like its prequel, focused on minimal relations $\twoalphalex \rightarrow (\tau)^\tau$. Since we have identified a whole class of order types which can consistently be the exponent in a partition relation on some $\twoalphalex$, namely finite sums of ordinals and reverse ordinals, the following problem (on which we have made some progress) is of interest:
\begin{problem}\label{largerhomog:problem}
    Classify the relation
    \[\twoalphalex \rightarrow (\sigma)^\tau\]
    for $\tau$ a finite sum of ordinals and reverse ordinals.
\end{problem}
The class of linear orders $\twoalphalex$ for $\alpha$ an ordinal is a very natural one to study, and indeed in \s{ZFC} \emph{every} linear order embeds order-preservingly into some $\twoalphalex$; we are, of course, working in \s{ZF} without Choice, by necessity, so it is not necessarily the case that every linear order embeds in some $\twoalphalex$. We therefore arrive at the following, extremely broad, project:
\begin{problem}\label{nontwoalpha:problem}
    Investigate the relation $\L \rightarrow (\sigma)^\tau_\chi$ for linear orders $\L$ which do not embed in any $\twoalphalex$.
\end{problem}
The first and second authors have proved some results in this direction, which will appear in \cite{structural}.

\subsection*{Acknowledgements}
This research was funded in whole or in part by the Austrian Science Fund (FWF) [10.55776/ESP5711024]. For open access purposes, the authors have applied a CC BY public copyright license to any author-accepted manuscript version arising from this submission. The third author was supported by the Italian PRIN 2022 ``Models, sets and classifications'', prot.\ 2022TECZJA.


\begin{thebibliography}{0000000}
        \bibitem[Ga25]{ieprsonr} L.\ A.\ Gardiner, \textit{Infinite-exponent partition relations on the real line}, arXiv:2507.12361 (2025)
        \bibitem[Gi53]{ginsburgremarks} S.\ Ginsburg, \textit{Some remarks on order types and decompositions of sets}, Transactions of the American Mathematical Society, Volume 74, no.\ 3 (1953), pp.\ 514--535
        \bibitem[GS$\infty$]{structural} L.\ A.\ Gardiner, J. Schilhan, \textit{Structural infinite-exponent partition relations and weak choice principles}, in preparation
        \bibitem[Ha08]{hausdorfforders} F.\ Hausdorff, \textit{Grundzüge einer Theorie der geordneten Mengen}, Mathematische Annalen, Volume 65 (1908), pp.\ 435--505
        \bibitem[Je73]{jechaxiomofchoice} T.\ J.\ Jech, \textit{The Axiom of Choice}, Studies in logic and the foundations of mathematics, Volume 75, North-Holland (1973)
        \bibitem[Ju68]{jullienthesis} P.\ Jullien, \textit{Contribution à l’étude des types d’ordres dispersés}, doctoral thesis, Université de Marseille (1968)
        \bibitem[Ka03]{kanamori} A.\ Kanamori, \textit{The Higher Infinite: Large Cardinals in Set Theory from Their Beginnings}, Springer, Berlin (2003)
        \bibitem[Kl77]{kleinberg} E.\ M.\ Kleinberg, \textit{Infinitary Combinatorics and the Axiom of Determinateness}, Springer-Verlag Lecture Notes in Mathematics, Volume 612 (1977)
        \bibitem[Lav71]{laver} R.\ Laver, \textit{On Fraïssé's order type conjecture}, Annals of Mathematics, Volume 93, no.\ 1 (1971), pp.\ 89--111
        \bibitem[LSW17]{crtolo} P.\ Lücke, P.\ Schlicht, T.\ Weinert, \textit{Choiceless Ramsey theory of linear orders}, Order, Volume 34 (2017), pp.\ 369--418
        \bibitem[Ma70]{mathias} A.\ R.\ D.\ Mathias, \textit{On a generalization of Ramsey's theorem}, doctoral thesis, University of Cambridge (1970)
        \bibitem[Mo59]{morelarithmetic} A.\ C.\ Morel, \textit{On the arithmetic of order types} Transactions of the American Mathematical Society, Volume 92, no.\ 1 (1959), pp.\ 48--71.
        \bibitem[Př76]{prikryadr} K.\ Příkrý, \textit{Determinateness and partitions}, Proceedings of the American Mathematical Society, Volume 54, no.\ 1 (1976), pp. 303--306
        \bibitem[Ro82]{rosenstein} J.\ G.\ Rosenstein, \textit{Linear Orderings}, Pure and Applied Mathematics, Academic Press, New York and London, Volume 98 (1982)
        \bibitem[So70]{solovay} R.\ M.\ Solovay, \textit{A model of set-theory in which every set of reals is Lebesgue measurable}, Annals of Mathematics, Volume 92, no.\ 1 (1970), pp.\ 1--56
    \end{thebibliography}
\end{document}